\documentclass[11pt]{article}
\usepackage{amssymb}
\usepackage{amsmath}
\usepackage{graphicx}

\newenvironment{proof}{\noindent {\bf Proof }}
{\hfill $\bullet$ \vspace{0.25cm}}

\def\E{{\mathbb E}}
\def\P{{\mathbb P}}
\def\R{{\mathbb R}}

\def\N{{\mathbb N}}
\def\1{{\mathbf 1}}

\def\F {{\mathcal F}}

\newtheorem{theo}{Theorem}
\newtheorem{prop}{\indent Proposition}

\newtheorem{rem}{\indent Remark}
\newtheorem{lem}{\indent Lemma}
\newtheorem{defin}{\indent Definition}
\newtheorem{cor}{\indent Corollary}

\newtheorem{ass}{\indent Assumption}

\title{Multi-class oscillating systems of interacting neurons}

\author{Susanne Ditlevsen \thanks{Department of Mathematical Sciences,
  University of Copenhagen, Universitetsparken 5, DK-2100 Copenhagen}\\
  \and
  Eva L{\"o}cherbach\thanks{
Universit{\'e} de Cergy-Pontoise, AGM UMR-CNRS 8088, 2 avenue Adolphe Chauvin, F-95302 Cergy-Pontoise Cedex} 
}

\begin{document}

\maketitle

\begin{abstract}
We consider multi-class systems of interacting nonlinear Hawkes
processes modeling several large families of neurons and study their mean field limits. As the total number of neurons goes to infinity we prove that the evolution within each class can be described by a nonlinear limit differential equation driven by a Poisson random measure, and state associated central limit theorems. We study situations in which the limit system exhibits oscillatory behavior, and relate the results to certain piecewise deterministic Markov processes and their diffusion approximations.
\end{abstract}



\section{Introduction}

Biological rhythms are ubiquitous in living organisms. The brain controls and helps maintain the internal clock for many of these rhythms, and fundamental questions are how they arise and what is their purpose. Many examples of such biological oscillators can be found in the classical book by Glass and Mackey (1988) \cite{GlassMackeyBook}. The motivation for this paper comes from the rhythmic scratch like network activity in the turtle, induced by a mechanical stimulus, and recorded and analyzed by Berg and co-workers \cite{Berg2007,Berg2008,Berg2013,Jahn2011}. Oscillations in a spinal motoneuron are initiated by the sensory input, and continue by some internal mechanisms for some time after the stimulus is terminated. 
While mechanisms of rapid processing are well documented in sensory systems, rhythm-generating motor circuits in the spinal cord are poorly understood. The activation leads to an intense synaptic bombardment of both excitatory and inhibitory input, and it is of interest to characterize such network activity, and to build  models which can generate self-sustained oscillations. 

The aim of this paper is to present a microscopic model describing a large network of interacting neurons which can generate oscillations. The activity of each neuron is represented by a point process, namely, the successive times at which the neuron emits an action potential or a so-called spike. A realization of this point process is called a spike train. It is commonly admitted that the spiking intensity of a neuron, i.e., the infinitesimal probability of emitting an action potential during the next time unit, depends on the past history of the neuron and it is affected by the activity of  other neurons in the network. Neurons interact mostly through chemical synapses, where a spike of a pre-synaptic neuron leads to an increase if the synapse is excitatory, or a decrease if the synapse is inhibitory, of the membrane potential of the post-synaptic neuron, possibly after some delay. In neurophysiological terms this is called synaptic integration. When the membrane potential reaches a certain upper threshold, the neuron fires a spike. Thus, excitatory inputs from the neurons in the network increase the firing intensity, and inhibitory inputs decrease it. Hawkes processes provide good models of this synaptic integration phenomenon by the structure of their intensity processes, see \eqref{eq:intensity0} below. We refer to Chevallier et al.\ (2015) \cite{ccdr}, Chornoboy et al.\  (1988) \cite{chorno},  Hansen et al.\ (2015)  \cite{hrbr} and to Reynaud-Bouret et al. (2014) \cite{pat} for the use of Hawkes processes in neuronal modeling. For an overview of point processes used as stochastic models for interacting neurons both in discrete and in continuous time and related issues, see also Galves and L\"ocherbach (2016) \cite{gl}.

In this paper, we study oscillatory systems of interacting Hawkes processes representing the time occurrences of action potentials of neurons. The system consists of several large populations of neurons. Each population might represent a different functional group of neurons, for example different hierarchical layers in the visual cortex, such as V1 to V4, or the populations can be pools of excitatory and inhibitory neurons in a network. Each neuron is characterized by its spike train, and the whole system is described by  multivariate counting processes $Z^N_{k, i} ( t), t \geq 0.$ Here, $ Z^N_{k, i} ( t)$ represents the number of spikes of the $i$th neuron belonging to the $k$th population, during the time interval $ [0, t ].$ The number of classes $n$ is fixed, and each class $k= 1, \ldots , n $ consists of $  N_k$ neurons. 
The total number of neurons is therefore $N = N_1 + \ldots + N_n.$ 

Under suitable assumptions, the sequence of counting processes $ ( Z^N_{k, i } )_{1 \le k \le n , 1 \le i \le N_k} $ is characterized by its intensity processes $ (\lambda^N_{k, i} ( t)  ) $ defined through the relation 
$$ \P ( Z^{N}_{k, i } \mbox{ has a jump in } ]t , t + dt ] | \F_t ) = \lambda^N_{k,i } (t)  dt , $$
where $ \F_t = \sigma (  Z^{N}_{k, i} (s)  , \, s \le t ,{1 \le k \le n , 1 \le i \le N_k} ) .$ We consider a mean-field framework where $ \lambda^N_{k, i } (t) $ is given by
\begin{equation}\label{eq:intensity0}
\lambda^N_{k,i} ( t)  = f_k  \left( \sum_{l=1}^n \frac{1}{N_l}\sum_{1 \le j \le N_l}     \int_{]0 , t [} h_{kl }( t-s) d Z^N_{ l,j} (s)  \right) .
\end{equation}
Here, $f_k : \R \to \R_+$ is the {\it spiking rate function} of population $k$, and $\{ h_{k l } : \R_+ \to \R \} $ is a family of {\it synaptic weight functions} modeling the influence of population $l $ on population $k$. By integrating over $]0,t[$ and not over $]-\infty,t[$, we implicitly assume initial conditions of no spiking activity before time 0. 

Equation (\ref{eq:intensity0}) has the typical form of the intensity of a multivariate nonlinear Hawkes process, going back to Hawkes (1971) \cite{Hawkes} and Hawkes and Oakes (1974) \cite{ho}. We refer to Br\'emaud and Massouli\'e (1996) \cite{bm} for the stability properties of multivariate Hawkes processes, and to Delattre, Fournier and Hoffmann (2015) \cite{dfh} and Chevallier (2015) \cite{chevallier} for the study of Hawkes processes in high dimensions.  

The structure of \eqref{eq:intensity0} is such that within each population, all neurons behave in a similar way, i.e., the intensity process $\lambda^N_{k,i} ( t) $ depends only on the empirical measures of each population. Thus, neurons within a given population are exchangeable. Therefore, we deal with a multi-class system of populations interacting in a mean-field framework which is reminiscent of Graham (2008) \cite{carl} and Graham and Robert (2009) \cite{carl2}.  Our aim is to study the large population limit when $ N  \to \infty$ and to show that in this limit self-sustained periodic behavior emerges even though each single neuron does not follow periodic dynamics. The study follows a long tradition, see e.g.\ Scheutzow (1985) \cite{Scheutzow1,Scheutzow2} in the framework of nonlinear diffusion processes, or Dai Pra, Fischer and Regoli (2015) \cite{paolo1} and Collet, Dai Pra and Formentin (2015) \cite{paolo2}. Our paper continues these studies within the framework of infinite memory point processes. The first important step is to establish propagation of chaos of the finite system $(Z^N_{k, i} ( t))_{1 \le k \le n , 1 \le i \le N_k} $ as $N \to \infty, $ under the condition that for each class $ 1 \le k \le n$,  $\lim_{N \to \infty} N_k/N$ exists and is in $]0, 1[$.

\subsection{Propagation of chaos}
In Section \ref{sec:propagation}, we study the limit behavior of the system $(Z^N_{k, i} ( t))_{1 \le k \le n , 1 \le i \le N_k} $ as $N \to \infty.$ We show in Theorem \ref{theo:one} that the system can be approximated by a system of inhomogeneous independent Poisson processes $ (\bar Z_1 (t) , \ldots ,\bar Z_n (t) ), $ where each $\bar Z_k (t)$ has intensity 
$$ f_k \left( \sum_{l=1}^n \int_0^t h_{kl } ( t-s)  d \E ( \bar Z_l  (s))  \right) dt. $$
Here, $\bar Z_k (t)$ represents the number of spikes during $[0,t]$ of a typical neuron belonging to population $k$ in the limit system. This result is an extension of results obtained by \cite{dfh} to the multi-class case. It follows that the system is {\em multi-chaotic} in the sense of \cite{carl}. The equivalence between the {\em chaoticity} of the system and a weak law of large numbers for the empirical measures, as proven in Theorem \ref{theo:one}, is well-known (see for instance Sznitman (1991) \cite{s}). This means that in the large population limit, within the same class, the neurons converge in law to independent and identically distributed copies of the same limit law. This property is usually called {\em propagation of chaos} in the literature. In particular, as pointed out in \cite{carl}, we have asymptotic independence between the different classes, and interactions between classes do only survive in law.


In Section \ref{sec:CLT}, still following ideas of \cite{dfh}, we state an associated central limit theorem in Theorem \ref{theo:2}. The extension to nonlinear rate functions $f_k$ requires the use of  matrix-convolution equations which go back to Crump (1970) \cite{crump} and Athreya and Murthy (1976) \cite{athreya-murthy} and which are collected in Appendix, Section \ref{sec:convolution}. 

\subsection{Oscillatory behavior of the limit system}
In Section \ref{sec:3} we present conditions under which the limit system possesses solutions which are periodic in law. To be more precise, the classes interact according to a cyclic feedback system and each class $k$ is only influenced by class $k+1,$ where we identify $n+1 $ with $1.$ In this case $m_t^k = \E (\bar Z_k (t)), 1 \le k \le n ,$ is solution of 
\begin{equation}\label{eq:limitmean0}
m_t^k = \int_0^t  f_k \left(  \int_0^s h_{k k+1} ( s- u ) d m_u^{k+1}   \right) ds .
\end{equation} 
If the memory kernels $ h_{k k+1}$ are given by Erlang kernels, as used e.g. in modeling the delay in the hemodynamics in nephrons, see \cite{Ditlevsen2005,KidneyDelay} and \eqref{eq:erlang} below, then  
Theorem \ref{theo:orbit} characterizes situations in which the system \eqref{eq:limitmean0} possesses attracting non-constant periodic orbits, that is, presents oscillatory behavior. This result goes back to deep theorems in dynamical systems, obtained by Mallet-Paret and Smith (1990) \cite{malletparet-smith} and used in a different context in Bena\"{\i}m and Hirsch (1999) \cite{michel}, from where we learned about these results. In particular, the celebrated Poincar\'e-Bendixson theorem plays a crucial role.

\subsection{Hawkes processes, associated piecewise deterministic Markov processes and longtime behavior of the approximating diffusion process}
Hawkes processes are truly infinite memory processes and techniques from the theory of Markov processes are in general not applicable. However, in the special situation where the memory kernels are given by Erlang kernels, the intensity processes can be described in terms of an equivalent high dimensional system of {\it piecewise deterministic Markov processes} (PDMPs). Once we are back in the Markovian world, we can study the longtime behavior of the process, ergodicity, and so on.

In Section \ref{sec:diffusion}, we obtain an approximating diffusion equation in \eqref{eq:cascadeapprox} which is shown to be close in a weak sense to the original PDMP defining the Hawkes process (Theorem \ref{theo:approx}). Once we dispose of this small noise approximation, we then study the longtime behavior in the case of two populations, $n=2.$ In particular, we show to which extent the approximating diffusion presents the same oscillatory behavior as the limit system. 
 
This approximating diffusion is highly degenerate having Brownian noise present only in two of its coordinates. However, the very specific {\em cascade structure} of its drift vector implies that the weak H\"ormander condition holds on the whole state space, and as a consequence, the diffusion is strong Feller. A simple Lyapunov argument shows that the process comes back to a compact set infinitely often, almost surely. 

Since the limit system possesses a non constant periodic orbit $\Gamma$ which is asymptotically orbitally stable, it is well known that there exists a local Lyapunov function $V( x) $ defined on a neighborhood of $\Gamma $ such that $V$ decreases along the trajectories of the limit system, describing the attraction of the limit system to $\Gamma $ (see e.g.\ Yoshizawa (1966) \cite{yoshizawa} and Kloeden and Lorenz (1986) \cite{kloeden}). This Lyapunov function is shown also to be a Lyapunov function for the approximating diffusion, in particular, the diffusion is also attracted to $\Gamma , $ once it has entered the basin of attraction of $\Gamma .$  A control argument shows finally that this happens infinitely often almost surely (Theorem \ref{theo:oscillation}), in particular, for large enough $N,$ the approximating diffusion also presents oscillations. 

We close our paper with some simulation studies.

\section{Systems of interacting Hawkes processes, basic notation and large population limits}
Consider $n$ populations, each composed by $ N_k  $ neurons, $k = 1 , \ldots , n .$ The total number of neurons in the system is $N = N_1  + \ldots + N_n $. The activity of each neuron is described by a counting process $ Z^{N}_{k, i} (t), {1 \le k \le n , 1 \le i \le N_k} , t \geq 0,$ recording the number of spikes of the $i$th neuron belonging to population $k$  during the interval $ [0, t ]. $ The sequence of counting processes $ ( Z^{N}_{k,i} ) $ is characterized by its intensity processes $ (\lambda^N_{k, i} ( t)  ) $ which are defined through the relation 
$$ \P ( Z^{N}_{k, i } \mbox{ has a jump in } ]t , t + dt ] | \F_t ) = \lambda^N_{k,i } (t)  dt , {1 \le k \le n , 1 \le i \le N_k} ,$$
where $ \F_t = \sigma (  Z^{N}_{k, i} (s)  , \, s \le t , {1 \le k \le n , 1 \le i \le N_k} ) $ and $ \lambda^N_{k, i } (t) $ are defined in \eqref {eq:intensity} below. We consider a mean field framework where $ N  \to \infty $ such that for each $ 1 \le k \le n ,$ 
$$\lim_{N \to \infty} \frac{N_k}{N} = p_k \, \mbox{ exists and is in } \, ]0, 1[ .$$
The intensity processes will be of the form
\begin{equation}\label{eq:intensity}
\lambda^N_{k,i} ( t)  = f_k  \left( \sum_{l=1}^n \frac{1}{N_l}\sum_{1 \le j \le N_l}     \int_{]0 , t [} h_{kl }( t-s) d Z^N_{ l,j} (s)  \right) ,
\end{equation}
where $f_k$ is the spiking rate function of population $k$ and where the $h_{k l }$ are memory kernels.

\begin{ass}\label{ass:1}
(i) All $f_k$ belong to $C^1 (\R; \R_+ )  .$\\
(ii) There exists a finite constant $L $ such that
for every $x$ and $x' $ in $\R,$ for every $ 1 \le k \le n ,$
\begin{equation}
\label{Lipsch-f}
|f_k (x)  -f_k (x')    |  \le L |x-x'| .
\end{equation}\\
(iii) The functions $ h_{kl}, 1 \le k , l \le n , $ belong to $ L^2_{loc} ( \R_+; \R ) .$ 
\end{ass}

\subsection{The setting}\label{subsec:setting}
We work on a filtered probability space $ ( \Omega, {\mathcal A}, \mathbb F)$ which we define as follows.  We write $\mathbb M$ for the canonical path space of simple point processes given by
\begin{multline*}
\mathbb M:= \{ {\tt m} = (t_n)_{n \in \N}:  t_1 > 0, \,  t_n \leq t_{n+1},t_n<t_{n+1} \,\,\, \mbox{if} \,\,\, t_n< + \infty , \, \lim_{n \to +\infty} t_n= + \infty\} .
\end{multline*}
For any ${\tt m} \in \mathbb M,$ any $n \in \N,$ let $T_n({\tt m})=t_n .$ We identify $ {\tt m} \in \mathbb M$ with the associated point measure $ \mu  = \sum_n \delta_{T_n ({\tt m })} $ and put ${\mathcal M}_t := \sigma \{ \mu (A) : A \in {\mathcal B} (\R), A \subset [0,t] \},$ $ {\mathcal M} = {\mathcal M}_\infty.$ 
Finally, we put $( \Omega, {\mathcal A}, \mathbb F):= ( \mathbb M, {\mathcal M}, ( {\mathcal M}_t )_{t \geq 0} )^{I}  $ where $I = \bigcup_{k=1}^n \{ (k, i ) , i \geq 1 \}.$ We write $ (Z^{N}_{k,i})_{{1 \le k \le n , 1 \le i \le N_k}  } $ for the canonical multivariate point measure defined on the finite dimensional subspace $( \mathbb M, {\mathcal M}, ( {\mathcal M}_t )_{t \geq 0} )^{I^N}$ of $ \Omega ,$ where $I^N= \bigcup_{k=1}^n \{ (k, i ) , 1 \le i \le N_k \}.$ 

\begin{defin}\label{def:1}[compare to Definition 1 of \cite{dfh}]. \\
A Hawkes process with parameters $ ( f_k, h_{kl}, 1 \le k, l \le n) $  is a probability measure $\P $ on $ ( \Omega, {\mathcal A}, \mathbb F)$ such that 
\begin{enumerate}
\item
$\P-$almost surely, for all $(k,i)\neq (l,j) ,  Z^{N}_{k, i } $ and $ Z^{N}_{l, j}  $ never jump simultaneously,
\item
for all $(k,i)\in I^N$, the compensator of $Z^{N}_{k,i}(t) $ is given by $\int_0^t \lambda^N_{k, i } (s )ds $ defined in \eqref{eq:intensity}. 
\end{enumerate}
\end{defin}

\begin{prop}
Under Assumption \ref{ass:1} there exists a path-wise unique Hawkes process \\$ (Z^{N}_{k,i} (t)_{(k,i) \in I^N} )$ for all $ t \geq 0.$  
\end{prop}

\begin{proof}
The proof is analogous to the proof of Theorem 6 in \cite{dfh}. 
\end{proof}

\subsection{Mean-field limit and propagation of chaos}\label{sec:propagation}
The aim of the paper is to study the process $(Z^N_{k, i } ( t))_{ (k,i) \in I^N } $ in the large population limit, i.e., as $N \to \infty.$ The convergence will be stated in terms of the empirical measures
\begin{equation}
\frac{1}{N_k} \sum_{1 \le i \le N_k} \delta_{ (Z^{N}_{k, i} (t))_{t \geq 0}}, \; 1 \le k \le n ,
\end{equation}
taking values in the set ${\cal P} ( D ( \R_+, \R_+   ) )$ of probability measures on the space of c\`adl\`ag functions, $  D ( \R_+, \R_+ ) .$ 
We endow $  D ( \R_+, \R_+ ) $ with the Skorokhod topology, and $ {\cal P} ( D ( \R_+, \R_+  ) ) $ with the weak convergence topology associated with the Skorokhod topology 
on $D ( \R_+, \R_+  ).$ 

Since we are dealing with multi-class systems, the classical notions of chaoticity and propagation of chaos have to be extended to this framework, see \cite{carl} for further details. We recall from \cite{carl} the following definition. Let $P_1, \ldots , P_n \in {\cal P} ( D ( \R_+, \R_+   ) ).$

\begin{defin}
The system $(Z^N_{k, i } ( t))_{ (k,i) \in I^N } $ is called $P_1 \otimes \ldots \otimes P_n-$multi-chaotic, if for any $ m \geq 1, $ 
$$ \lim_{ N\to \infty } {\cal L} \left( (Z^N_{ k, i } ), 1 \le k \le n , 1 \le i \le m \right)  =  P_1^{\otimes m} \otimes \ldots \otimes P_n^{\otimes m} .$$
\end{defin}

In particular, Corollary 5.2 of \cite{carl2} shows that in this case we have convergence in distribution
$$ \frac{1}{N_k} \sum_{1 \le i \le N_k} \delta_{ (Z^{N}_{k, i} )} \stackrel{\cal L}{\to } P_k ,$$
as $N \to \infty ,$ for any $ 1 \le k \le n.$ The limit measure $P_k$ has to be understood as the distribution of the limit process $ \bar Z_k (t) $, where the associated limit system
is given by
\begin{equation}\label{eq:limit}
\bar Z_k (t) = \int_0^t \int_{\R_+} \1_{\{ z \le  f_k ( \sum_{l=1}^n \int_0^s h_{kl} ( s- u ) d \E ( \bar Z_l (u) ) \}} N^k ( ds, dz) , 
\end{equation} 
$1 \le k \le n ,$ where $ N^k $ are independent Poisson random measures (PRMs) on $ \R_+ \times \R_+ $ each having intensity measure $ ds dz $. 

Introduce $ m_t = (m_t^1, \ldots , m_t^n)=( \E (\bar Z_1 (t),\ldots ,  \bar Z_n (t))) .$ Taking expectations in \eqref{eq:limit}, it follows that $m_t $ is solution of 
\begin{equation}\label{eq:limitmean}
m_t^k = \int_0^t  f_k \left( \sum_{l=1}^n \int_0^s h_{kl} ( s- u ) d m_u^l   \right) ds , 1 \le k \le n .
\end{equation}  

\begin{theo}\label{theo:one}
Under Assumption \ref{ass:1}, there exists a path-wise unique solution to \eqref{eq:limit} such that $t \mapsto  \E ( \sum_{k=1}^n \bar Z_k (t)  ) $ is locally bounded. Moreover, the system of processess $(Z^N_{k, i } )_{ (k,i) \in I^N } $ is $P_1 \otimes \ldots \otimes P_n-$multi-chaotic, where $P_k = {\cal L} ( \bar Z_k  ),$ $1 \le k \le n.$ In particular, for any $ i \geq 1,$ 
$$( ( Z^N_{1, i } (t), \ldots , Z^N_{n, i }(t) )_{t \geq 0} )  \stackrel{\cal L}{\to } (( \bar Z_1 ( t) , \ldots , \bar Z_n ( t) )_{t \geq 0}) $$
as $N \to \infty $ (convergence in $ D ( \R_+, \R_+^n ), $ endowed with the Skorokhod topology). 
\end{theo}

\begin{rem}
The above theorem shows that any fixed finite sub-system is asymptotically independent with neurons of class $k$ having the law of $\bar Z_k .$
\end{rem}

The proof of Theorem \ref{theo:one} is a direct adaptation of the proof of Theorem 8 in \cite{dfh} to the multi-class case. 

\begin{proof}
1) Let $( \bar Z_1 (t), \ldots , \bar Z_n (t) )$ be any solution of \eqref{eq:limit} and consider the associated vector $ m_t = (m_t^1, \ldots ,m_t^n)=( \E (\bar Z_1(t),\ldots ,  \bar Z_n(t))) .$ Then $m_t $ is solution of \eqref{eq:limitmean}, and an easy adaptation of Lemma 24 of \cite{dfh} shows that this equation has a unique non-decreasing (in each coordinate) locally bounded solution, which is of class $C^1 .$ 

2) Well-posedness and uniqueness of a solution satisfying that $t \mapsto  \E ( \sum_{k=1}^n \bar Z_k (t)  ) $ is locally bounded follow then as in \cite{dfh}, proof of Theorem 8.

3) Propagation of chaos: Let $N^N_{k,i} (ds, dz ) , (k,i)\in I^N,  $ be i.i.d.\ PRMs having intensity $ds dz $ on $ \R_+ \times \R_+ $. For each $N \geq 1,  $ consider the Hawkes process $ (Z^N_{ k,i} (t))_{ (k,i)\in I^N,t \geq 0}  $ given by 
\begin{eqnarray*}
Z^N_{k, i } (t)  &=& \int_0^t \int_0^\infty \1_{\left\{ z \le \lambda_{k,i}^N(s) \right \} } N^N_{k,i} ( ds ,dz) \\
&=& \int_0^t \int_0^\infty \1_{\left\{ z \le f_k  \left( \sum_{l=1}^n \frac{1}{N_l} \sum_{ 1 \le j \le  N_l}\int_0^{s-} h_{kl} ( s-u) d Z^N_{l, j } (u)   \right)\right\} } N^N_{k,i} ( ds ,dz) .
\end{eqnarray*}
Indeed, $Z^N_{k,i } $ defined in this way is a Hawkes process in the sense of Definition \ref{def:1}, as follows from Proposition 3 of \cite{dfh}. We now couple $Z^N_{k,i }$ with the limit process \eqref{eq:limit} in the following way.
Let $m_t $ be the unique solution of \eqref{eq:limitmean}. Put
\begin{equation}\label{eq:versionlimitprocess}
 \bar Z^N_{k,i} (t)  = \int_0^t \int_0^\infty \1_{\left\{ z \le f_k  \left(  \sum_{l=1}^n \int_0^{s} h_{kl} ( s-u) d m_u^l  \right)\right\} }N^N_{k,i}  ( ds ,dz),
\end{equation}
where $N^N_{k, i }$ is the PRM driving the dynamics of $Z^N_{k, i }.$ Obviously, for all $1 \le i \le N_k,$ $\bar Z^N_{k, i } \stackrel{\cal L}{=} \bar Z_k .$ Moreover, the limit processes $ \bar Z^N_{k, i } , 1 \le k \le n, 1 \le i \le N_k,$ are independent. 

Denote  $ \Delta^N_{k,i} ( t) = \int_0^t | d [ \bar Z^N_{k,i} (u)  - Z^N_{k, i } (u) ] | ,$ and $\delta^N_{k,i} (t) = \E ( \Delta^N_{k,i}  ( t)) .$ Notice that this last quantity does not depend on $ i \in \{ 1, \ldots , N_k \} ,$ due to the exchangeability of the neurons within one class. 
Then
$$ \sup_{ u \in [0, t ] } | \bar Z^N_{k, i }(u)   - Z^N_{ k , i } (u)  | \le \Delta^N_{k,i} ( t) , \mbox{ whence } \E \left[ \sup_{ u \in [0, t ] } | \bar Z^N_{k,i} (u)   - Z^N_{ k , i } (u)  | \right] \le \delta^N_{k,1}(t) := \delta^N_k (t) . $$

We start by controlling $ \Delta^N_{k,i}  (t) $ which is given by 
\begin{multline*}
\Delta^N_{k,i}  (t) = \int_0^t \int_0^\infty \left | \1_{ \left \{ z \le f_k  \left(  \int_0^{s}\sum_{l=1}^n  h_{kl} ( s-u) d m_u^l  
 \right) \right \} } \right . \\ \left . - \1_{\left \{ z \le f_k  \left(  \sum_{l=1}^n \frac{1}{N_l} \sum_{ 1 \le j \le  N_l} \int_0^{s-} h_{kl} ( s-u) d Z^N_{l, j } (u)    \right) \right \} } 
\right  | N^N_{k,i} ( ds, dz ) .
\end{multline*}
Using the Lipschitz continuity of $f_k $ with Lipschitz constant $L, $ 
\begin{multline}\label{eq:delta}
\frac{1}{L}\E ( \Delta^N_{k,i}  (t) ) \le 
 \int_0^t   \sum_{l=1}^n \E \left | \frac{1}{N_l}  \sum_{ 1 \le j \le N_l }  \int_0^{s-} h_{ kl} ( s-u) (d m_u^l - d \bar Z^N_{ l,j} (u))  \right | ds
 \\
+   \int_0^t \sum_{l=1}^n   \E \left |  \frac{1}{N_l} \sum_{ 1 \le  j \le N_l }  \int_0^{s-} h_{ kl} ( s-u)  d [ \bar Z^N_{l,j} (u)  - Z^N_{l , j } (u) ] \right | ds
=: A + B ,
\end{multline}
where $A$ denotes the terms of the RHS of the first line, and $B$ the terms within the second line. Now, using Lemma 22 of \cite{dfh},
$$
 B \le \int_0^t \E  \int_0^{s-} \left[ \sum_{l=1}^n | h_{kl } ( s-u )| d \Delta^N_{ l, 1} ( u )  \right] ds
\le \int_0^t  \left[ \sum_{l=1}^n | h_{ kl } (t-u )| \delta^N_l (u)  \right]  du . 
$$
To control $A, $ let $X^N_{k,l,j} (t)  = \int_0^{t-}  h_{ kl} ( t-u) d \bar Z^N_{l, j } (u) , $ for $ 1 \le j \le N_l . $ Then $ X^N_{k,l,j} (t), 1 \le j \le N_l, $ are i.i.d.\ having mean $\int_0^t  h_{ kl} ( t-u) d m_u^l .$ Hence
$$ A \le \sum_{l=1}^n \frac{1}{\sqrt{N_l}} \int_0^t \sqrt{ Var ( X_{k,l, 1}^N (s) ) }ds .$$
But
$$ X^N_{k,l,1} (s) = \int_0^{s-} \int_0^\infty \1_{\left \{ z \le f_{l} ( \sum_{m=1}^n \int_0^u  h_{lm} (u-r) d m^m_r  ) \right \} } h_{kl} ( s-u) N^N_{l, 1} (du, dz) ,$$
and thus, since the integrand is deterministic, 
$$
 X^N_{k,l,1} (s)  - \E ( X^N_{k,l,1} (s)  ) = 
 \int_0^{s-} \int_0^\infty \1_{ \left \{ z \le f_{l} ( \sum_{m=1}^n \int_0^u  h_{lm} (u-r) d m^m_r  ) \right \}} h_{kl} ( s-u) \tilde N^N_{l, 1} (du, dz) , 
$$
where $ \tilde N^N_{l, 1}(ds, dz)  = N^N_{l, 1 } (ds, dz) - ds dz $ is the compensated PRM. Recalling \eqref{eq:limitmean} we deduce that
$$
 Var (X^N_{k,l,1} (s)) =  \int_0^s f_{l} \left( \sum_{m=1}^n \int_0^u  h_{lm} (u-r) d m^m_r  \right) h^2_{kl} ( s-u) du 
=  \int_0^s h^2_{kl} ( s-u) d m^l_u  .
$$

Putting $ \|\delta^N(t)\|_1 = \sum_{k=1}^n \delta^N_k(t)  , \; \|m_t\|_1 =\sum_{k=1}^n  m^k_t , \;  \|h(t)\|_1  =\sum_{k,l=1}^n| h_{kl}(t)|  , $ we obtain
\begin{multline}\label{eq:217}
\frac{1}{L}  \|\delta^N(t)\|_1   \le   \int_0^t  \| h ( t-u ) \|_1 \; \| \delta^N (u) \|_1 du \\
+ \left (  \sum_{k=1}^n \frac{1}{\sqrt{N_k}} \right )   \int_0^t  \left( \int_0^s   \|h(s-u)\|_1^2 d \|m_u\|_1 \right)^{1/2}   ds . 
\end{multline}
It follows, as in Step 4 of the proof of Theorem 8 in \cite{dfh}, that 
$$ \sup_{ t \le T } \|\delta^N(t)\|_1\le C_T \left ( \sum_{k=1}^n \frac{1}{\sqrt{N_k}} \right ) .$$
Consequently, for any fixed $(k,i)\in I^N , $ as $ N \to \infty ,$  
\begin{equation}\label{eq:finalestimate}
 \E \left[ \sup_{ u \in [0, T ] } | \bar Z^N_{k,i} (u)   - Z^N_{ k , i } (u)  | \right] \le C_T \left (\sum_{k=1}^n \frac{1}{\sqrt{N_k}} \right ) \to 0 .
\end{equation}

The end of the proof is now standard, based on arguments developed in \cite{carl} and \cite{carl2}. Recall that neurons within a given population are exchangeable. 
Then, by the proof of Theorem 5.1 in \cite{carl2}, in order to prove propagation of chaos, it is enough to 
show that for each fixed sequence $\ell_k, 1 \le k  \le n , $ 
$$ ( (Z^N_{1, 1} (t) )_{t \geq 0}, \ldots, (Z^N_{1, \ell_1} (t) )_{t \geq 0} , \ldots  , (Z^N_{ n, 1} (t) )_{t \geq 0}, \ldots , (Z^N_{n, \ell_n} ( t) )_{ t \geq 0}) $$ 
goes in law to $ \ell_1$ independent copies of $ \bar Z_1, \ldots,  $ and $ \ell_n $ independent copies of $\bar Z_n $ (convergence in $D ( \R_+, \R_+^{ \ell_1 + \ldots + \ell_n})$). Since the topology of uniform convergence on compact time intervals is finer than the Skorokhod topology, 
this follows clearly from \eqref{eq:finalestimate}, and thus the proof is finished.
\end{proof}

\section{Central limit theorem}\label{sec:CLT} 
A natural question to ask is to which extent the large time behavior of the limit system $(m_t^1, \ldots , m_t^n ) $ predicts the large time behavior of the finite size system, in particular in the case when the limit system presents oscillations (see Section \ref{sec:3} below). To answer this question, the present section states a central limit theorem where convergence of both $N $ and $ t $ to infinity is considered. All proofs can be found in Appendix.

First we control the longtime behavior of the limit system represented by its integrated intensities $ (m_t^1, \ldots , m_t^n ).$ It is well-known that linear Hawkes processes, i.e., the case when the rate functions $f_k$ are linear, can be described in terms of classical Galton-Watson processes. In the non-linear case, a comparison with a Galton-Watson process is still possible if the rate functions are Lipschitz. In our case, the associated offspring matrix is given by $ \Lambda := (\Lambda_{i j } )_{ 1 \le i, j \le n } ,$ where
\begin{equation}\label{eq:Lambda}
\Lambda_{ij } = L \int_0^\infty | h_{i j}  ( t)| dt , 1 \le i , j \le n, 
\end{equation}
with $L$ given in \eqref{Lipsch-f}. Define the matrix 
\begin{equation}\label{eq:matrixh}
 H(t) = \Big( L |h_{ik } ( t)| 
\Big)_{ 1 \le i, k \le n }, \mbox{ for any } t \geq 0,
\end{equation}
such that
$ \Lambda = \int_0^\infty H(t) dt.$

Classically, one distinguishes the subcritical, the critical and the supercritical cases. Since we only need to bound the intensities, we concentrate on the subcritical and the supercritical case.
The subcritical case is defined by the following property of the matrix $\Lambda.$ 

\begin{ass}\label{ass:sc}
The functions $ h_{kl} , 1 \le k , l \le n ,$ belong to $ L^1 ( \R_+; \R) \cap L^2 ( \R_+; \R)  ,$ and the largest eigenvalue $\mu_1$ of $ \Lambda $ is strictly smaller than $1.$  
\end{ass}
We then obtain the following bound on the growth of $m_t^k$.
\begin{prop}\label{cor:subc}
Grant Assumptions \ref{ass:1} and \ref{ass:sc}. Then there exists a constant $\alpha_0$ such that
$$  m_t^k    \le \alpha_0 t , \quad \mbox{ for all } 1 \le k \le n . $$
Moreover, as in \cite{dfh}, Remark 9, in this case there exists a constant $C$ such that
\begin{equation}\label{eq:rem9first}
 \E ( \sup_{s \le t } | Z^N_{k, i } (s) - \bar Z^N_{k,i} ( s)  |) \le C t N^{-1/2} , 
\end{equation}
for $1 \le k \le n$, where $ \bar Z^N_{k, i } $ is defined in \eqref{eq:versionlimitprocess}.
\end{prop}

In the supercritical case, the control on the growth of $m_t^k$ is more tricky. We need the following assumption.

\begin{ass}\label{ass:superc}
The functions $ h_{kl} , 1 \le k , l \le n ,$ belong to $ L^1 ( \R_+; \R) $ and there exist $ p \geq 1 $ and a constant $C,$ such that $ | h_{kl} (t) | \le C ( 1 + t^p ) $ for all $t \geq 0, $ for any $  1 \le k, l \le n. $ 
Moreover, the largest eigenvalue $\mu_1$ of $ \Lambda $ in \eqref{eq:Lambda} is strictly larger than $1.$ 
\end{ass}

\begin{prop}\label{cor:superc}
Grant Assumptions \ref{ass:1} and \ref{ass:superc}. Then for all $1\leq k\leq n$,
$$ m_t^k    \le c e^{\alpha_0 t} , $$
where  $ \alpha_0 $ is unique such that $ \int_0^\infty e^{ - \alpha_0 t } H(t) dt $ has largest eigenvalue $\equiv +1.$ Here, $H(t)$ is given in \eqref{eq:matrixh}. Moreover, for $1 \le k \le n $ and for some constant $C,$
\begin{equation}\label{eq:rem9}
 \E (\sup_{s \le t }  | Z^N_{k, i } (s) - \bar Z^N_{k,i} ( s)  |) \le C e^{ \alpha_0 t} N^{-1/2} .
\end{equation}
\end{prop}

We obtain the following central limit theorem. It is an extension of Theorem 10 of \cite{dfh} to the nonlinear case and several populations. 
\begin{theo}\label{theo:2}
Grant Assumption \ref{ass:1} and either Assumption \ref{ass:sc} or \ref{ass:superc}. Suppose moreover that 
 for all $1 \le k \le n$, $ \lim \inf_{ t \to \infty } m_t^k/t  \geq \alpha_k $  for some $\alpha_k  > 0.$  We will consider limits as both $N $ and $ t $ tend to infinity, under the constraints that $ t/N \to 0 $ in the subcritical case and that $ e^{\alpha_0 t} t^{-1} N^{-1/2} \to 0$ in the supercritical case, where $\alpha_0$ is given in Proposition \ref{cor:superc}.  
\begin{enumerate}
\item
For any fixed $i  , $ we have that $ Z^N_{k, i }(t) /m_t^k $ tends to $1$ in probability. More precisely,  
$$\lim \sup_{N , t \to \infty } ( m_t^k)^{1/2} \E [ | Z^N_{k, i}(t) /m_t^k - 1 |] \le C ,$$ 
for some constant $C$ and for $ k=1,\ldots , n.$  
\item
For any fixed $ \ell_1 , \ldots , \ell_n , $ the vector
$$ \left( \left (  \frac{Z^N_{1, i }(t) -m_t^1}{\sqrt{ m_t^1 }}  \right )_{ 1 \le i \le \ell_1 },\ldots , \left (  \frac{Z^N_{n, i }(t) -m_t^n}{\sqrt{ m_t^n }}  \right )_{ 1 \le i \le \ell_n }\right) $$ tends in law to ${\cal N}( 0, I_{\ell_1 +\ldots + \ell_n}) $ as $(t, N) \to \infty , $ under the constraint that $\displaystyle{\lim_{N \to \infty } }N_k / N > 0$ for all $1 \le k \le n$. 
\end{enumerate}
\end{theo}

The proof of Theorem \ref{theo:2} and Propositions \ref{cor:subc} and \ref{cor:superc} can be found in Appendix.

\begin{rem}
Since the rate functions are nonlinear, we only obtain the central limit theorem  in the regime $ t N^{-1} \to 0 $ (subcritical case) or $ N^{-1/2} t^{-1} e^{\alpha_0 t }  \to 0 $ (supercritical case),  contrarily to \cite{dfh} who do not have any restriction in the subcricital case and who only impose $ N^{-1} e^{\alpha_0 t } \to 0 $ in the supercritical case. This is due to the fact on the one hand that we deal with the nonlinear case and on the other hand that we do not dispose of general asymptotical equivalents of $ t \mapsto m_t^k .$ 

\end{rem}

\section{Oscillations  and associated dynamical systems in monotone cyclic feedback systems}\label{sec:3}
The aim of this section is to study the limit system \eqref{eq:limit} and \eqref{eq:limitmean} and describe situations in which oscillations will occur. Throughout this section we suppose that the information is transported through the system  according to a {\it monotone cyclic feedback system} \cite{malletparet-smith}. That it is monotone means that the rate functions $f_k$ are non-decreasing, and a cyclic feedback system means that each population $k$ is only influenced by population $k+1, $ where we identify $n+1$ with $1.$ Thus, $h_{kl} \equiv 0 $ for all $ k, l $ such that $ l \neq k+1 .$ 
The memory kernels $ h_{k k+1}$ describe how population $k+1$ influences population $k.$ 

From now on, we identify $m^{n+1} $ with $m^1$ and 
introduce the memory variables
\begin{equation}\label{eq:memory}
 x^k_t  = \int_0^t h_{kk+1 }  (t-s)  d m_s^{k+1} , \mbox{ for } \, 1 \le k \le  n  .
 \end{equation}
We have $ m_t^k = \int_0^t f_k ( x^{k}_s) ds .$ For specific choices of kernel functions the above system of memory variables can be developed into a system of differential equations without delay by increasing the dimension of the system, see \eqref{eq:cascade} below. We call this a {\em Markovian cascade} of successive memory terms. 
It is obtained by using Erlang kernels, given by 
\begin{equation}\label{eq:erlang}
 h_{k k+1 } (s) =  c_k e^{ - \nu_k s } \frac{ s^{\eta_k}}{\eta_k!}  , \mbox{ for }  k < n ,  \mbox{ and }  h_{n 1} ( s) = c_n e^{-\nu_n s} \frac{ s^{\eta_n}}{\eta_n !}   ,
\end{equation}
where $\eta_k \in \N_0 , c_k \in \{-1,1\} $ and $ \nu_k > 0$ are fixed constants. Here, $\eta_k+1$ is the order of the delay, i.e., the number of differential equations needed for population $k$ to obtain a system without delay terms. The delay of the influence of population $k+1$ on population $k$ is distributed and taking its maximum absolute value at $\eta_k/\nu_k$ time units back in time, and the mean is $(\eta_k+1)/\nu_k$ (if normalizing to a probability density). The higher the order of the delay, the more concentrated is the delay around its mean value, and in the limit of $\eta_k \rightarrow \infty$ while keeping $(\eta_k+1)/\nu_k$ fixed, the delay converges to a discrete delay. The sign of $c_k$ indicates if the influence is inhibitory or excitatory.

Observing that $h_{kk+1} ' (t) = - \nu_{k}h_{kk+1}  (t) +  c_{k} \, \frac{ t^{\eta_{k}- 1}}{(\eta_{k}-1)! }  e^{- \nu_{k}  t}$ leads to the following auxiliary variables 
$$ x_t^{k,l} = \int_{0}^t c_k e^{ -\nu_k  (t-s)} \frac{ (t-s)^{\eta_k-l}}{(\eta_k-l)!} d m_s^{ k+1} , 1 \le k \le n , 0 \le l\le \eta_k,$$
where we identify $x^k =x^{k,0}.$ Then we can rewrite 
\begin{equation}\label{eq:tobeusedlater}
 \frac{d x_t^{k,l}}{dt} =  - \nu_k x_t^{k,l} + x_t^{k,l+1} , l < \eta_k .
\end{equation}
Iterating this argument, the following system of coupled differential equations is obtained. For all $1 \le k \le n,$ and where as usual $n+1$ is identified with $1$, 
\begin{eqnarray}\label{eq:cascade}
\frac{d x_t^{k,l}}{dt} &=& - \nu_k x^{k,l}_t +  x^{k, l+1}_t , \, \, 0 \le l  < \eta_k  , \nonumber \\
\frac{d x_t^{k, \eta_k}}{dt} &=& - \nu_k x^{k, \eta_k}_t  + c_k  f_{k+1} ( x^{k+1, 0}_t) ,
\end{eqnarray}
with initial conditions $x^{k,l}_0=0$. System \eqref{eq:cascade} exhibits the structure of a monotone cyclic feedback system as considered e.g.\ in \cite{malletparet-smith} or as (33) and (34) in \cite{michel}. If $\prod_{k=1}^n c_k > 0, $ then the system \eqref{eq:cascade} is of total positive feedback, otherwise it is of negative feedback. We obtain the following simple first result.

\begin{prop}\label{prop:xstar}
Suppose that $ \prod_{k=1}^n c_k  < 0 $ and that $f_1, \ldots , f_n $ are non-decreasing. Then \eqref{eq:cascade} admits a unique equilibrium $x^* .$ 
\end{prop}

\begin{proof}
Any equilibrium $x^* $ must satisfy 
$$(x^*)^{n, \eta_n}  = \frac{c_n}{\nu_n} f_1 \circ \frac{c_1}{\nu_1^{\eta_1 +1}}  f_2 \circ \ldots \circ \frac{c_{n-1}}{\nu_{n-1}^{\eta_{n-1} +1}}  f_n ( \frac{1}{\nu_n^{\eta_n} }(x^*)^{n,\eta_n} ) .$$ 
Since $\frac{c_n}{\nu_n} f_1 \circ \frac{c_1}{\nu_1^{\eta_1 +1}}  f_2 \circ \ldots \circ \frac{c_{n-1}}{\nu_{n-1}^{\eta_{n-1} +1}}  f_n ( \frac{1}{\nu_n^{\eta_n} } \; \cdot  )   $ is decreasing, there exists exactly one solution $(x^*)^{n,\eta_n} $ in $\R$.  Once $ (x^*)^{n, \eta_n} $ is fixed, we obviously have $ (x^*)^{n, \eta_n - 1  } = \frac{1}{\nu_n} (x^*)^{n, \eta_n}, $ and the values of the other coordinates of $x^* $ follow in a similar way.
\end{proof} 

In special cases system \eqref{eq:cascade} is necessarily attracted to a non-equilibrium periodic orbit. Let $ \kappa := n + \sum_{ k=1}^n \eta_k$ be the dimension of \eqref{eq:cascade}.  

We introduce the following assumption. 

\begin{ass}\label{ass:4}
Suppose that $f_k, 1 \le k \le n , $ are non-decreasing bounded analytic 
functions. Moreover, suppose that  $\rho := \prod_{k=1}^n c_k f_k' ((x^*)^{k,
0}) $ satisfies that $ \rho < 0.$
\end{ass}

Notice that under Assumption \ref{ass:4}, the conditions of Proposition \ref{prop:xstar} are satisfied, and thus \eqref{eq:cascade} admits a unique equilibrium $x^*  $ under Assumption \ref{ass:4}.

The following theorem is based on Theorem 4.3 of \cite{malletparet-smith} and generalizes the result obtained by Theorem 6.3 in \cite{michel}.

\begin{theo}\label{theo:orbit}
Grant Assumption \ref{ass:4}. Consider all solutions $\lambda
$ of

\begin{equation}\label{eq:racinesunite}
(\nu_1 + \lambda)^{\eta_1 +1} \cdot \ldots \cdot (\nu_n + \lambda)^{\eta_n
+1} = \rho
\end{equation}

and suppose that there exist at least two solutions $\lambda$ of
\eqref{eq:racinesunite} such that

\begin{equation}\label{eq:unstable}
\mbox{Re } ( \lambda ) > 0.
\end{equation}

(i) $x^* $ is linearly unstable, and the system \eqref{eq:cascade}
possesses at least one, but no more than a finite number of periodic
orbits. At least one of them is orbitally asymptotically stable. \\

(ii) Moreover, if $ \kappa = 3 ,$ then there exists a globally
attracting invariant surface $\Sigma $ such that $x^* $ is a repellor
for the flow in $\Sigma .$ Every solution of \eqref{eq:cascade} will
be attracted to a non constant periodic orbit.

\end{theo}

\begin{proof}
Since all functions $f_k$ are bounded, the system \eqref{eq:cascade}
possesses a compact invariant set $K.$ Rewriting \eqref{eq:cascade} as
$ \dot x = F (x) , $ where $x = (x^{1, 0}, \ldots , x^{n, \eta_n})^T$, the characteristic polynomial $ P ( \lambda ) $ of $ D F (
x^*)$ is given by

$$ P ( \lambda ) =\prod_{k=1}^n (-\nu_k - \lambda)^{\eta_k +1} -
(-1)^{\kappa } \rho = (-1)^\kappa [\prod_{k=1}^n (\nu_k +
\lambda)^{\eta_k +1} - \rho ]. $$

By assumption, there exist at least two eigenvalues having strictly
positive real part. Therefore $x^* $ is unstable.
Moreover, since $\rho < 0 ,$
$$ det ( - DF ( x^*)) > 0 $$
which is condition (4.5) of 
\cite{malletparet-smith}. Then, the last assertion of item (i) follows
from Theorem 4.3 of \cite{malletparet-smith}.

To prove part (ii), we follow the proof of Theorem 6.3 in \cite{michel}. First, notice that $
D F ( x^*) $ is given by the matrix

$$ \left( \begin{array}{lll}
-d & a& 0\\
0&-e & b\\
c & 0 & -f
\end{array}
\right) ,$$

where $d,e,f > 0 $ and $\rho=abc < 0.$ Hence, either all $a,b,c$ are
negative or only one of them, say $c,$ is negative. In the first case,
$- D F ( x^*) $ is a positive irreducible matrix, in the second case,
the change of variables $ y_1=x_1, y_2=- x_2, y_3= x_3 $ leeds to a negative
irreducible matrix. We therefore suppose without loss of generality that we are in the first
case. Then the Perron-Frobenius theorem implies that $ D F ( x^*) $
possesses a single largest eigenvalue which is strictly negative, and
the eigenvector associated to it has all its components of the same
sign. Moreover, the other two eigenvectors associated to the conjugate
complex eigenvalues having positive real part do not have all
components of the same sign.

By Theorem 1.7 of Hirsch (1988) \cite{hirsch88a}, there exists a
globally attracting invariant surface $\Sigma $ such that every
trajectory within the invariant set $K$ is eventually attracted to
$\Sigma.$ By \cite{michel} the equilibrium $x^*$ is a repellor for the flow in $K.$
Hence, the Poincar\'e-Bendixson theorem implies that each such
trajectory will eventually converge to a non constant periodic orbit.

\end{proof}

\begin{rem}
If $ \nu_1 = \ldots = \nu_n = \nu $, 
then for $\kappa \geq 3$ the following condition 
\begin{equation}\label{eq:unstable2}
|\rho | > \frac{\nu^\kappa}{ \left( 
\cos ( \frac{\pi }{\kappa })\right)^{ \kappa }} 
\end{equation}
implies \eqref{eq:unstable}. Indeed, the different
eigenvalues for $1 \le j \le \kappa$ are given by
$$ \lambda_j = - \nu - |\rho|^{\frac{1}{\kappa}} e^{i \frac{2 j \pi}{\kappa}
} \, (\kappa \mbox{ odd}) \quad ; \quad \lambda_j = - \nu + |\rho|^{\frac{1}{\kappa}} e^{i \frac{(2 j-1) \pi}{\kappa}
} \, (\kappa \mbox{ even}) .$$
If $\kappa$ is odd, then there is exactly one real root, which is strictly negative, $ \lambda_{\kappa} = - \nu - |\rho|^{\frac{1}{\kappa}}$. The rest are complex conjugate pairs with real part $- \nu - |\rho|^{\frac{1}{\kappa}} \cos(\frac{ 2 j \pi}{\kappa})$. The maximal value is $- \nu + |\rho|^{\frac{1}{\kappa}} \cos(\frac{\pi}{\kappa})$ for $j=(\kappa \pm 1)/2$, such that \eqref{eq:unstable2} implies \eqref{eq:unstable}. If $\kappa$ is even, then all roots are complex conjugate pairs with real part $- \nu + |\rho|^{\frac{1}{\kappa}} \cos(\frac{ (2 j -1)\pi}{\kappa})$. The maximal value is as before, now for $j=1,\kappa$, such that again \eqref{eq:unstable2} implies \eqref{eq:unstable}.

\end{rem}

\begin{rem}[Phase transition due to increasing memory]

In some cases, increasing the order of the memory, i.e.\ the value of some of the exponents $\eta_k$ in \eqref{eq:erlang} or equivalently the value of $\kappa$, 
can lead to a phase transition within the system \eqref{eq:cascade}. At the phase transition point, a
system which was stable can become unstable, and in certain cases, increasing the order even more might stabilize the system again. As an example, consider a family of $n$ populations of neurons, where $n>1$ is fixed, and such that $\nu_k =\nu $ for all $ 1 \le k \le n$. If  $\kappa=2$, the fixed point is stable since eigenvalues are $\lambda_j = -\nu \pm i \sqrt{|\rho | }$, and only damped oscillations occur. We will assume $\kappa \geq 3$. 

First note that $\rho$ is bounded due to the Lipschitz condition on the rate functions $f_k$. The right hand side of \eqref{eq:unstable2} goes to infinity for $\nu \to \infty$ for all values of $\kappa$, and thus, if $\nu$ is large, the system will always be stable and not exhibit oscillations. For any fixed value of $\nu > 1$, it also goes to infinity for $\kappa \to \infty$, such that a possible unstable system becomes stable for increasing $\kappa$. This implies that for a discrete delay of any value the system will never exhibit oscillations, since a discrete delay is obtained for $\eta_k \rightarrow \infty$, keeping $\eta_k/\nu_k$ constant.

Now assume that $\nu=1$. Then increasing $\kappa$ does not change the
coordinates $(x^*)^{k,0}$ of the equilibrium state $x^*$, so $\rho$ does not change. The right hand side of \eqref{eq:unstable2} decreases towards one, so if $-8<\rho < -1$, then there exists $\kappa_0 > 3 $ minimal such
that for all $ \kappa \geq \kappa_0, $ \eqref{eq:unstable2} is fulfilled, 
but $|\rho | \le  \left( \nu/
\cos ( \frac{\pi }{\kappa })\right)^{ \kappa } $ for $\kappa < \kappa_0.$
Then all models corresponding to $ \kappa < \kappa_0$ have $x^* $ as
attracting equilibrium point, but for $ \kappa \geq \kappa_0,$ the
equilibrium $x^* $ becomes unstable. 

\end{rem}

As a corollary of the above Theorem, we show that one of the conditions needed to state the central limit theorem in Theorem \ref{theo:2} is satisfied.

\begin{cor}
Suppose that $n=2$ and that the conditions of Theorem \ref{theo:orbit} hold true. Then there exist $ \alpha_1, \alpha_2> 0$ such that 
$$ \lim \inf_{t \to \infty } \frac{m_t^k }{t}\geq \alpha_k , \; k=1, 2 .$$
\end{cor}

\begin{proof}
Since $c_1 c_2 < 0, $ any solution to \eqref{eq:cascade} is eventually attracted to a non constant periodic orbit. Since $m_t^k = \int_0^t f_k ( x_s^{k, 0} ) ds $ and since $f_k$ is non-decreasing and strictly positive, it follows that $\lim \inf_{t \to \infty } \frac{1}{t} m_t^k > 0.$  

\end{proof}

\section{Study of an approximating diffusion process and simulation study}\label{sec:diffusion}
In this section we work with the cyclic feedback system of the last section. The aim is to study to which extent the behavior of the limit system is also observed within the finite size system $ Z^N_{k, i } .$

\subsection{An associated system of piecewise deterministic Markov processes}\label{sec:erlang}
Introducing the family of adapted c\`adl\`ag processes (recall \eqref{eq:memory})
\begin{equation}\label{eq:intensity2}
X^N_{k} (t) :=\frac{1}{N_{k+1}} \sum_{ j = 1 }^{N_{k+1}}  \int_{]0, t]} h_{k k+1 } (  t- s ) d Z^N_{k+1, j } ( s)  = \int_{]0, t]} h_{kk+1 } (t-s) d \bar Z^N_{k+1} ( s) ,
\end{equation}
where $ \bar Z^N_{k+1} ( s) = \frac{1}{N_{k+1}} \sum_{j=1}^{N_{k+1}} Z^N_{k+1, j } ( s)  $ and recalling \eqref{eq:intensity},  it is clear that the dynamics of the system is entirely determined by the dynamics of the processes $ X^N_{k } ( t- ) , t \geq 0 .$\footnote{We have to take the left-continuous version, since intensities are predictable processes.} In some sense, $ X^N_{k } $ describes the accumulated memory belonging to the directed edge pointing from population $k+1$ to population $k.$ Without assuming the memory kernels to be Erlang kernels, the system $(X^N_{k } , 1 \le k \le n  ) $ is not Markovian. For general memory kernels, Hawkes processes are truly infinite memory processes. 

When the kernels are Erlang, given by \eqref{eq:erlang},
taking formal derivatives in \eqref{eq:intensity2} with respect to time $t$ and introducing for any $ k $ and $ 0  \le l \le \eta_{k } $  
\begin{equation}
 X^N_{k, l} ( t) :=  c_{k} \, \int_{]0, t]}   \frac{ (t-s)^{\eta_k- l}}{(\eta_k-l)! }  e^{- \nu_{k} ( t-s)} d \bar Z^N_{k+1} ( s)  ,
\end{equation} 
we obtain the following system of stochastic differential equations which is a stochastic version of \eqref{eq:cascade}.
\begin{equation}\label{eq:cascadepdmp}
\left\{ 
\begin{array}{lcl}
d X^N_{k, l } (t) &=& [ -  \nu_{k} X^N_{k, l  } ( t) + X^N_{k, l+1} (t) ] dt , \; 0 \le l < \eta_k ,  \\
d X^N_{k, \eta_k} (t) &=& -  \nu_{k} X^N_{k,  \eta_k} (t) dt + c_{k} d \bar Z^N_{k+1} (t) .
\end{array}
\right. 
\end{equation}
Here, $X^N_{k }$ is identified with $X^N_{k, 0 }, $ $\bar Z^N_{k} = \frac{1}{N_k} \sum_{j=1}^{N_k} Z_{k, j }^N, $ and each $Z^N_{k, j } $ jumps at rate $f_{k} (  X^N_{k, 0 } (t- ) ) .$ We call the system \eqref{eq:cascadepdmp} a {\em cascade of memory terms}. 
Thus, the dynamics of the Hawkes process $ (Z^N_{k, i } (t))_{ (k,i)\in I^N}$ is entirely determined by the piecewise deterministic Markov process (PDMP) $ ( X^N_{k, l })_{(1\le k \le n, 0 \le l \le \eta_k )}  $ of dimension $ \kappa$.

\subsection{A diffusion approximation in the large population regime}
The process $\bar Z^N_{k+1} (t)$ appearing in the last equation of \eqref{eq:cascadepdmp} jumps at a rate given by $ N_{k+1} f_{k+1}( X^N_{k+1, 0 } (t- )) ,$ having jumps of size $  \frac{1}{N_{k+1} }.$ Its variance is $ \frac{f_{k+1}( X^N_{k+1, 0 } (t- ))}{N_{k+1}}.$ 
Therefore, it is natural to consider the approximating diffusion process
\begin{equation}\label{eq:cascadeapprox}
\left\{ 
\begin{array}{lcl}
 d Y^N_{k, l } (t) &=& [ -  \nu_{k} Y^N_{k, l  } ( t) + Y^N_{k, l +1} (t) ] dt , \;  \quad \quad 0 \le l < \eta_k ,  \\
d Y^N_{k,  \eta_k} (t) &=& -  \nu_{k} Y^N_{k,  \eta_k} (t) dt + c_{k} f_{k+1} (  Y^N_{ k+1, 0} (t) ) dt  +  c_{k} \frac{\sqrt{f_{k+1} (  Y^N_{ k+1, 0} )(t)  }}{ \sqrt{ N_{k+1}}} d B_{ k+1} (t) ,
\end{array}
\right . 
\end{equation}
where the $B_{l} (t), 1 \le  l \le n ,$ are independent standard Brownian motions, approximating the jump noise of each population. Write $ A^X $ for the infinitesimal generator of the process \eqref{eq:cascadepdmp} and $A^Y$ for the corresponding generator of \eqref{eq:cascadeapprox}. Moreover, write $ P_t^X $ and $P_t^Y$ for the associated Markovian semigroups. We denote generic elements of the state space $ \R^\kappa $ of $ Y^N$ by  
$ x = ( x^1, \ldots , x^\kappa  ) .$ Finally, for a function $ g$ defined on $\R^\kappa , $ we define 
$$ \| g\|_{r, \infty } := \sum_{k=0}^r \sum_{ |\alpha | = k } \| \partial^\alpha g \|_\infty .$$
Then we obtain the following approximation result showing that $Y^N $ is a good small noise approximation of $X^N.$

\begin{theo}\label{theo:approx}
Suppose that all spiking rate functions $f_k$ belong to the space $   C^5_b $ of bounded functions having bounded derivatives up to order $5.$ Then there exists a constant $C$ depending only on $f_1, \ldots , f_n $ and the bounds on its derivatives such that for all $ \varphi \in C^4_b (\R^\kappa ; \R ) , $ 
$$ \| P_t^X \varphi - P_t^Y \varphi \|_\infty \le C t \frac{\| \varphi\|_{4, \infty}}{N^2 } .$$
\end{theo}

The proof is given in the Appendix. 

Theorem \ref{theo:approx} is a first step towards convergence in law and shows that the diffusion process \eqref{eq:cascadeapprox} is a good approximation of \eqref{eq:cascadepdmp}, as $N \to \infty .$ 
However,  in the limit of $N \to \infty $, \eqref{eq:cascadeapprox} is not a diffusion anymore, since the diffusive term tends to zero. Both processes, $X^N $ and $Y^N$, tend to the limit process described in section \ref{sec:propagation}. This convergence is of rate $ \frac{1}{N}, $ which is slower than the approximation proved in Theorem \ref{theo:approx}. 

\subsection{Oscillations of the approximating diffusion at fixed population size}
We now show to which extent the approximating diffusion process \eqref{eq:cascadeapprox} imitates the oscillatory behavior of 
the limit system described in section \ref{sec:3}. Consider two populations, $n=2$, where the memory kernels are given by \eqref{eq:erlang}. 

We denote by 
\begin{equation}\label{eq:drift}
b(x) := \left( 
\begin{array}{c} 
- \nu_1 x^{ 1 } + x^{2 } \\
- \nu_1 x^{2} + x^{3} \\
\vdots\\
- \nu_1 x^{\eta_1 +1} + c_1 f_2 ( x^{\eta_1 +2} ) \\
- \nu_2 x^{\eta_1+2} + x^{\eta_1+3 } \\
\vdots \\
- \nu_2 x^{\kappa } + c_2 f_1 ( x^{1} ) 
\end{array}
\right) 
\end{equation}
the drift vector of \eqref{eq:cascadeapprox}. Moreover, we introduce the $ \kappa \times 2-$diffusion matrix 
\begin{equation}\label{eq:diff}
\sigma (x) := \left( 
\begin{array}{cc}
0&0\\
\vdots & \vdots \\
0 & \frac{c_1}{\sqrt{p_2}} \sqrt{ f_2 ( x^{\eta_1+2 } ) } \\
0 & 0 \\
\vdots & \vdots \\
\frac{c_2}{\sqrt{ p_1} } \sqrt{ f_1 ( x^{1 } ) } & 0 
\end{array}
\right) ,
\end{equation}
where $p_1 = N_1 / N , p_2 = N_2/ N. $ Then we may rewrite \eqref{eq:cascadeapprox} as 
\begin{equation}\label{eq:diffusionsmallnoise}
 d Y^N (t)  = b ( Y^N (t) ) dt + \frac{1}{ \sqrt{N}} \sigma ( Y^N (t) ) d B(t) , 
\end{equation} 
with $ B(t) = ( B^1(t), B^2(t) )^T$. 

Throughout this section, we assume the conditions of Theorem \ref{theo:orbit}, in particular, suppose that  $ c_1 c_2 < 0$. Moreover, suppose that $f_1 $ and $f_2$ are smooth strictly positive non-decreasing functions. In this case, under condition \eqref{eq:unstable}, the associated limit system possesses a non constant periodic orbit which is asymptotically orbitally stable. We will now show that also the finite size system \eqref{eq:diffusionsmallnoise} is attracted to this periodic orbit.

The existence of a global Lyapunov function implies that there exists a compact set $K $ such that process \eqref{eq:diffusionsmallnoise} visits $K$ infinitely often, almost surely. More precisely, recalling that $A^Y$ denotes the infinitesimal generator of \eqref{eq:diffusionsmallnoise}, we have the following result. In order to simplify notation, in the next proposition, we write $ x = (x^{1,0}, \ldots, x^{ 1, \eta_1}, x^{2, 0}, \ldots, x^{2 , \eta_2} ) $ for generic elements of $\R^\kappa .$ 

\begin{prop}\label{prop:lyapunovglobal}
Grant the conditions of Theorem \ref{theo:orbit}. Let $ j(x) $ be a smoothed version of $ |x|, $ i.e.,  $j(x) = |x| $ for all $ |x| \geq 1 ,$  $|j '(x) | \le C , |j'' (x) | \le C ,$ for some constant $C,$ for all $x.$ Put then $ G (x) :=  \sum_{k=1}^2\sum_{l=0}^{\eta_k}\frac{l+1}{\nu_k^{l} } j( x^{k,l} ) . $ Then $G$ is a Lyapunov-function for \eqref{eq:diffusionsmallnoise} in the sense that  
$$ A^Y G  ( x) \le - c G  (x) + d ,$$
for some constants $c, d > 0$ depending on  $ \max_k \|f_k\|_\infty .$ 
\end{prop}

\begin{proof}
The above statement follows from the fact that the drift part of $A^Y G ( x) $ is given by 
$$
 \sum_{k=1}^2 \sum_{ j = 0 }^{\eta_k - 1} \left\{ - \nu_k x^{k, j } + x^{k, j +1} \right\} \frac{\partial G}{\partial x^{k, j }}  - \sum_{k=1}^2 \nu_k x^{k, \eta_k } \frac{\partial G}{\partial x^{k, \eta_k }} + \sum_{k=1}^2 c_k f_{k+1} ( x^{k+1, 0 } ) \frac{\partial G}{\partial x^{k, \eta_k }}  .$$
But for $ |x^{k, j } | \geq 1,$ 
\begin{multline*}
\left\{ - \nu_k x^{k, j } + x^{k, j +1} \right\} \frac{\partial G}{\partial x^{k, j }}  
= \left\{ - \nu_k x^{k, j } + x^{k, j +1} \right\} sign ( x^{k,j}) \frac{  (j+1)}{(\nu_k)^{j}} \\
\le  - \frac{ ( j+1) }{\nu_k^{j-1} } | x^{k, j } | + \frac{ ( j+1) }{\nu_k^{j} } | x^ {k, j +1 }  | .
\end{multline*}
As a consequence, if $ |x^{k, j}| \geq 1 $ for all $k, j,$ the contribution of the drift part can be upper bounded by 
$$
 -   \sum_{k=1}^2  \nu_k |x^{k, 0} | -  \sum_{k=1}^2\sum_{j=1}^{\eta_k } \frac{1}{\nu_k^{j-1}} |x^{k, j }  | + 2 \max ( \frac{\eta_1\! +\! 1}{\nu_1^{\eta_1}}\| f_1\|_\infty , \frac{\eta_2\! +\! 1}{\nu_2^{\eta_2}}\| f_2\|_\infty)  \le - c G(x) + d ,
$$
for some constants $c , d > 0 .$ On the other hand, the contributions coming from terms with $ | x^{k, j } | \le 1 $ are bounded, and the contribution coming from the diffusion part of $A^Y G $ is bounded as well. This finishes the proof.
\end{proof}

In particular, putting $K = \{ G \le 2d/c\} , $ it follows that 
\begin{equation}\label{eq:driftcond0}
 A^Y G (x) \le - \frac{c}{2} G  (x) + d \1_K ( x) .
 \end{equation}
It is well-known (see e.g.\ Douc, Fort and Guillin (2009) \cite{DFG}) that (\ref{eq:driftcond0}) implies that  
\begin{equation}\label{eq:moment}
\E_x (e^{ \frac{c}{2}  \tau_K}) \le  G(x) ,
\end{equation}
where $\tau_K = \inf \{ t \geq 0 : Y^N ( t) \in K\} .$ 
Thus, the process comes back to the compact $K$ infinitely often, and excursions out of $K$ have exponential moments. In particular, we can concentrate on the study of the trajectories inside $K.$

We will now study the behavior of the trajectories of $Y^N$ inside the compact $K$. By Theorem \ref{theo:orbit}, under condition \eqref{eq:unstable}, the limit system possesses a non constant periodic orbit which is asymptotically orbitally stable. Denote this orbit by $\Gamma $ and let $ T$ be its periodicity. We suppose without loss of generality that $ \Gamma \subset K.$ In the following, we will show that each time the process is inside $K,$ it will also visit vicinities of the periodic orbit (in a sense that will be made precise in Theorem \ref{theo:oscillation} below). We start with support properties. 

Fix $\varepsilon > 0$ and let $ S (\varepsilon ,  \Gamma ) := \{ x : d (x, \Gamma ) < \varepsilon \} $ be a tube around this orbit. Denote by $ Q_x^{Y}$ the law of the solution $ (Y^N(t), t \geq 0) $ of  \eqref{eq:diffusionsmallnoise}, starting from $Y^N (0) = x .$ Fix $ t_1 > 1 $ and let $ O = \{\varphi  \in C ( \R_+; \R^\kappa ) : \varphi(t) \in  S( \varepsilon, \Gamma ) \; \forall \;  1  \le t \le t_1 \} .$ Then we have the following first result concerning the support of $ Q_x^{Y}.$

\begin{prop}\label{prop:control}
Under the conditions of Theorem \ref{theo:orbit} and supposing that $ f_1 $ and $f_2$ are strictly positive, the following holds true. For all $x \in \R^\kappa ,$ we have that  $ Q_x^{Y} ( O) > 0 .$ 
\end{prop}
Hence, the diffusion $Y^N$ visits the tube around $ \Gamma ,$ starting from any initial point $x \in \R^\kappa ,$ during the time interval $ [1, t_1],$ with positive probability. The fact that we choose time intervals $[1, t_1 ] $ is not important, and we could equally work with any time interval $ [ t_0, t _1 ] , $ for any fixed $ t_0 > 0 ,$ see the proof in Appendix \ref{sec:controlproof}. 

Note that the above proposition gives a statement concerning the support of the law of $Y^N  $ for fixed $N.$  Its proof relies on the support theorem for diffusions. The result does not say anything about the actual value of the probability of tubes around $ \Gamma ,$ nor about the precise $N \to \infty -$asymptotics : this is outside the scope of the present paper and will be the subject of a future work. 

Proposition \ref{prop:control} implies that there is strictly positive probability to visit tubes around $\Gamma, $ yet, we do not know that such visits arrive almost surely, within a finite time horizon. In order to prove this, we will show that $ x \mapsto Q_x^N ( O) $ is lower-bounded on compacts. Thus, 
we have to control the dependence on the starting configuration $x$ in the above proposition. Therefore, we show that $ x \mapsto Q_x^Y ( A) $ is continuous for Borel sets $ A \in {\cal B} ( C ( \R_+; \R^\kappa ) ) $ which are e.g.\ of the form $ O $, i.e.,  we need to show that the process $Y^N$ is strong Feller. However, $Y^N$ is a degenerate diffusion process since Brownian noise only appears in the coordinates $Y^N_{1, \eta_1} $ and $Y^N_{2, \eta_2}.$ Following Ishihara and Kunita (1974) \cite{Ich-Ku}, Lemma 5.1, we therefore show that the process satisfies the weak H\"ormander condition. Notice that due to the specific form of the diffusion matrix in equation \eqref{eq:diff}, the It{\^o} and the Stratonovich forms are the same. 

Recall that for smooth vector fields $f(x)$ and $ g( x) : \R^\kappa   \to \R^\kappa , $ the Lie bracket $ [f, g ] $ is defined by
$$ [f, g ]^i = \sum_{j = 1}^\kappa  \left( f^j \frac{ \partial g^i }{\partial x^j } - g^j \frac{ \partial f^i }{\partial x^j } \right) \; , \; i = 1, \ldots , \kappa . $$
 We write $\sigma^1 , \sigma^2 : \R^\kappa \to \R^\kappa $ for the two column vectors constituting the diffusion matrix $ \sigma $ of \eqref{eq:diff}.

\begin{defin}\label{def:hoer}
Define a set ${\cal L}$ of vector fields by the `initial condition' $\sigma^1 , \sigma^2  \in {\cal L}$ and  an arbitrary number of iteration steps 
\begin{equation}\label{eq:iteration}
L\in {\cal L} \;\Longrightarrow \; [ b,L] , [ \sigma^1  , L], [\sigma^2 , L ]  \in {\cal L}  \;.   
\end{equation}
For $M \in \N$, define the subset ${\cal L}_M$ by the same initial condition and at most $M$ iterations (\ref{eq:iteration}). Write ${\cal L}_M^*$ for the closure of ${\cal L}_M$ under Lie brackets; finally, write 
$$
\Delta_{{\cal L}_M^*}\;:=\;\mbox{\rm LA}({\cal L}_M)
$$
for the linear hull of ${\cal L}_M^*$, i.e.\ the Lie algebra spanned by ${\cal L}_M$. 
\end{defin}

\begin{defin}
We say that a point $z^* \in \R^\kappa $ is of {\em full weak H\"ormander dimension} if there is some $M \in \N$ such that   
\begin{equation}\label{fwHd}
({\rm dim}\; \Delta_{{\cal L}^*_M})(z^*) \; =\; \kappa  . 
\end{equation}
\end{defin}

Due to the cascade structure of the drift vector $b$ and since $f_1 ( \cdot ) , f_2 (\cdot) > 0 $ on $ \R,$  it is straightforward to show that the weak H\"ormander dimension holds at all points $x \in \R^\kappa:$

\begin{prop}
Suppose that $f_1 $ and $f_2$ are smooth and strictly positive. Then for all $x \in \R^\kappa , $ $({\rm dim}\; \Delta_{{\cal L}^*_M})(x) \; =\; \kappa  ,$ where $ M = \max ( \eta_1, \eta_2 ) .$ 
\end{prop}

\begin{proof}
The proof is done by first calculating the Lie-bracket $ [ \sigma^1, b] $ and $ [\sigma^2 , b ]$ and then successively bracketing with $b.$ 
\end{proof}

Once the weak H\"ormander condition holds everywhere, it follows that the process is strongly Feller, see  \cite{Ich-Ku}. As a consequence, the following holds.

\begin{cor}\label{cor:density}
Let $A \in {\cal B} ( C ( \R_+ ; \R^\kappa ) ) $ be of the type $ A = \{ \varphi : \varphi (t) \in B \; \forall \; t_1 \le t \le t_2 \}  ,$ for some $ B \in {\cal B} ( \R^\kappa )$ and for $t_1<t_2$. Then 
$$ x \mapsto Q_x^Y (A) \mbox{ is continuous.} $$
\end{cor}

\begin{proof} 
Using the Markov property, we have
$$ Q^Y_x (A) = \int P^Y_{t_1} (x, dy ) Q^Y_y ( \{ \varphi : \varphi (t) \in B \; \forall \; 0  \le t \le t_2- t_1  \}) = P_{t_1}^Y \Phi ( x) , $$
where $ \Phi ( y) = Q^Y_y ( \{ \varphi : \varphi (t) \in B \; \forall \; 0  \le t \le t_2- t_1  \})$ is measurable and bounded. But $ 
 P_{t_1}^Y \Phi ( x) $ is continuous in $x, $ since $Y^N$ is strong Feller.
\end{proof}

A direct consequence of the above result is the fact that 
$$ K \ni x \mapsto Q_x^Y ( O) $$
is strictly lower bounded for any fixed $t_1 >  1 .$  This fact enables us to conclude our discussion and to show that the diffusion approximation \eqref{eq:diffusionsmallnoise} will have the same type of oscillations as the limit system $ (m_t^1, m_t^2 ).$ Let 
$$ \tau_{\Gamma} ( t_1) := \inf \{ t \geq 0 : Y^N ( s) \in S( \varepsilon, \Gamma )\; \forall \: t \le s \le t+ t_1 \} .$$ 

\begin{theo}\label{theo:oscillation}
Grant the assumptions of Theorem \ref{theo:orbit} and let $\Gamma $ be a non constant periodic orbit of period $T$ of the limit system which is asymptotically orbitally stable. Then for all $\varepsilon > 0 $ and $ t_1  > 1 ,$  there exist $ C, \lambda > 0 $ such that 
$$ \E_x ( e^{ \lambda \tau_\Gamma ( t_1)} ) \le C G  ( x)  .$$
Moreover, 
$$ \lim \sup_{t \to \infty } \1_{ \{Y^N ( s) \in S( \varepsilon, \Gamma )\; \forall \: t \le s \le t+ t_1 \}}  = 1 $$
$\P_x-$almost surely, for all $x \in \R^\kappa .$
\end{theo} 

In particular, the above theorem implies that the process $Y^N$ visits the oscillatory region $ S( \varepsilon, \Gamma )$ during time intervals of length $t_1 $ infinitely often. The choice of $t_1 $ in the above theorem is free. By choosing $t_1 $ larger than the period $T$ of $\Gamma $ (or even choosing $t_1 \geq k T $ for some fixed $k \in \N$), 
this implies that $Y^N$ will present oscillations infinitely often almost surely. 

\begin{proof}
By Proposition \ref{prop:control}, for every $x  $ and every fixed $t_1 > 1,  $  $ Q_x^Y ( O ) > 0 .$ By continuity and since $K$ is compact, this implies that  
$$ \inf_{ x \in K } Q_x^Y ( O)   > 0 .$$
  
Now, by \eqref{eq:moment}, the process visits the compact set $K$ infinitely often, almost surely, with exponential moments for the successive visits of the process to $K.$ The assertion then follows by the conditional version of the Borel-Cantelli lemma. 
\end{proof}

\begin{rem}
The above result shows that $Y^N$ does actually visit tubes around the periodic orbit infinitely often, with waiting times that possess exponential moments. It is valid for a fixed population size $N.$ 
The above result does not show that, starting from a vicinity of $ \Gamma , $ the diffusion stays there for a long time, before being kicked out of the tube due to noise. Such a study is much more difficult, it is related to the large deviation properties of the process  and will be the topic of a future work. 
\end{rem}

\subsection{Simulation study}

In this Section we will check by simulations how $N$ influences the approximating diffusion, and compare the behavior of \eqref{eq:cascade} and \eqref{eq:cascadeapprox}. We set $n=2$ and $c_1=-1, c_2=1$ such that population 1 is inhibitory and population 2 is excitatory, and thus $\rho < 0$, and it is a negative feedback system. We will use the following bounded, Lipschitz and strictly increasing rate functions,
\begin{eqnarray*}
f_1(x) = \left \{ \begin{array}{cc}
10 e^x & \mbox{for } x<\log (20) \\
\frac{400}{1+400e^{-2x}} & \mbox{for } x\geq \log (20) \\
\end{array} \right . &;&
f_2(x) = \left \{ \begin{array}{cc}
e^x & \mbox{for } x<\log (20) \\
\frac{40}{1+400e^{-2x}} & \mbox{for } x\geq \log (20) \\
\end{array} \right . .
\end{eqnarray*}
In the first set of simulations, we put $\nu_1=\nu_2=1$ and 
$\eta_1=3, \eta_2=2$ such that $\kappa=7$. Then $(x^*)^{1,l} = -2.424$ for $l = 0, 1, \ldots , \eta_1$, and $(x^*)^{2,l} = 0.885$ for $l = 0, 1, \ldots , \eta_2$. This yields $\rho = -2.15$ and $\nu^\kappa/ \left( 
\cos ( \frac{\pi }{\kappa })\right)^{ \kappa } =2.08$, and thus, \eqref{eq:unstable2} is fulfilled. The period is approximately $2\pi / \omega=12.98$, where $\omega = |\rho |^{1/\kappa} \sin (\pi/\kappa)  $. Finally, we put $N_1=N_2=20$. Results are presented in Fig.\ \ref{fig:1}. The cascade structure in the memory variables is clearly seen, and the noisy diffusion approximation follows the limit cycle. The periodic behavior is evident.

\begin{figure}[th!]
\centering\includegraphics[width=7cm]{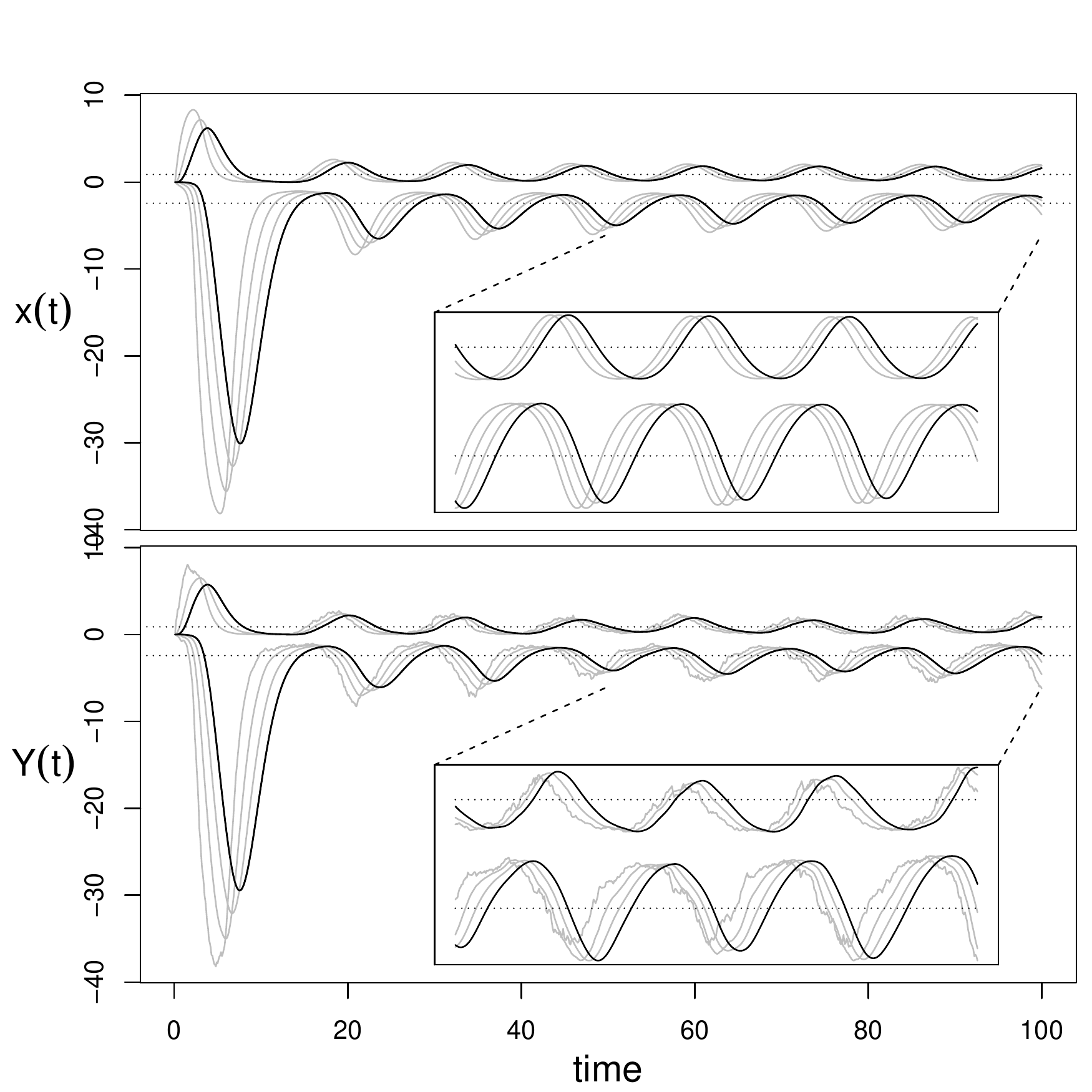} \, \includegraphics[width=7cm]{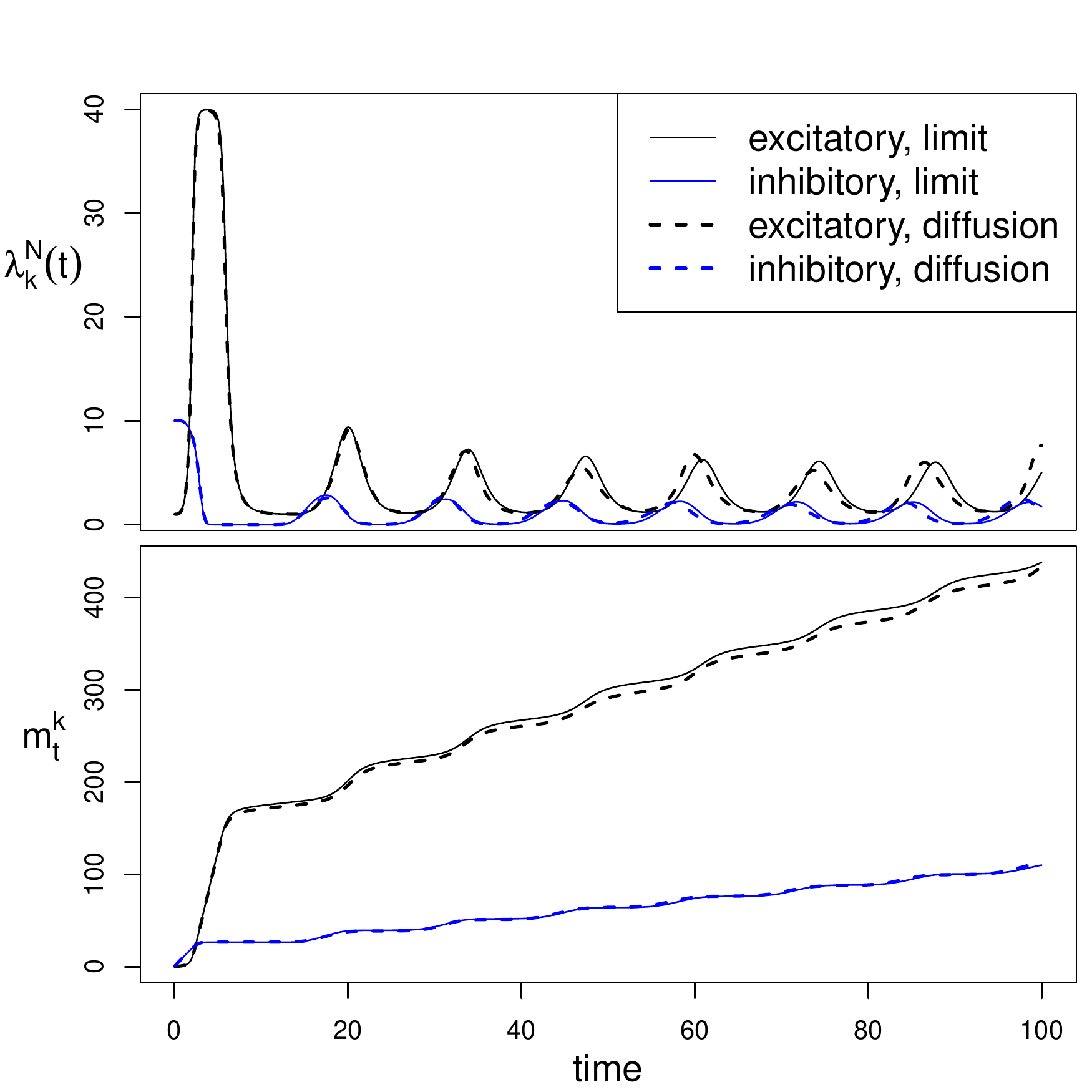}
\caption{ \label{fig:1} Comparison of the limit system and the diffusion approximation through simulations. Left top panel: The limiting system \eqref{eq:cascade}. The black curves are $x^{k,0}$ given in \eqref{eq:memory}, and those that are felt by the system. The gray curves are the variables in the Markovian memory cascade. The dotted lines indicate the location of the critical point. The inset is a blow-up of the last part of the simulation. Parameters are $c_1=-1, c_2=1, \nu_1=\nu_2=1,\eta_1=3, \eta_2=2$. Left bottom panel: As above but for the approximating system \eqref{eq:cascadeapprox}, with $N_1=N_2=20$. Right top panel: Intensity processes corresponding to the simulations on the left. Right bottom panel: The cumulative intensity processes. 
As predicted by the Central Limit Theorem \ref{theo:2}, both approximations are getting worse as $ t$ increases. 
}
\end{figure}

To investigate how close the approximating diffusion follows the oscillatory behavior of the limit system, we simulated 20 repetitions of the noisy process for different values of $N=20,100,200,1000$ and for $p_1=p_2=1/2$, and compared it to the limit system on a later time interval in Fig.\ \ref{fig:2}. For small $N$, the system shifts phase randomly relative to the limiting system, but it maintains the oscillations. For larger $N$, the system follows the limiting system closely on a longer time horizon.

\begin{figure}[th!]
\centering\includegraphics[width=7cm]{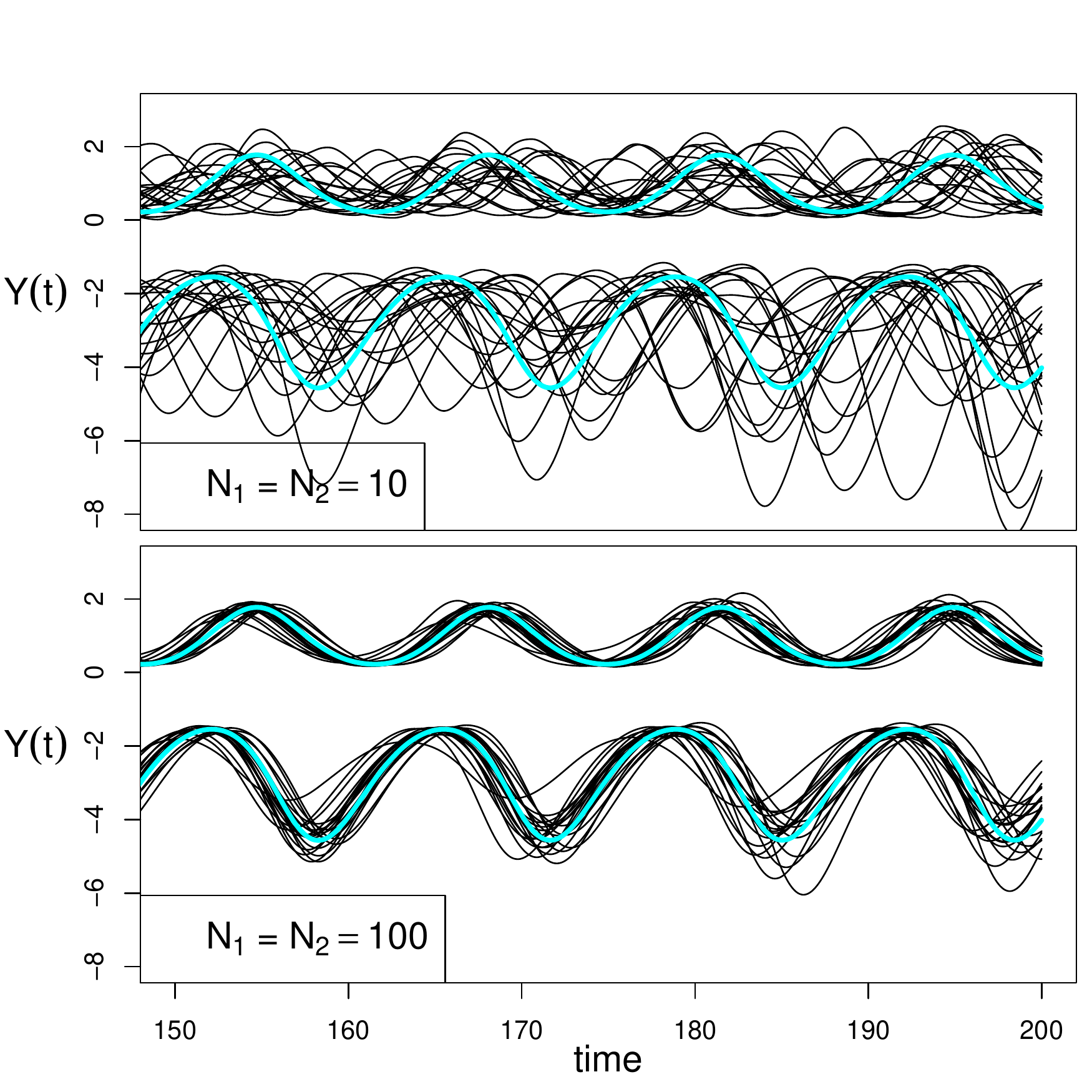} \, \includegraphics[width=7cm]{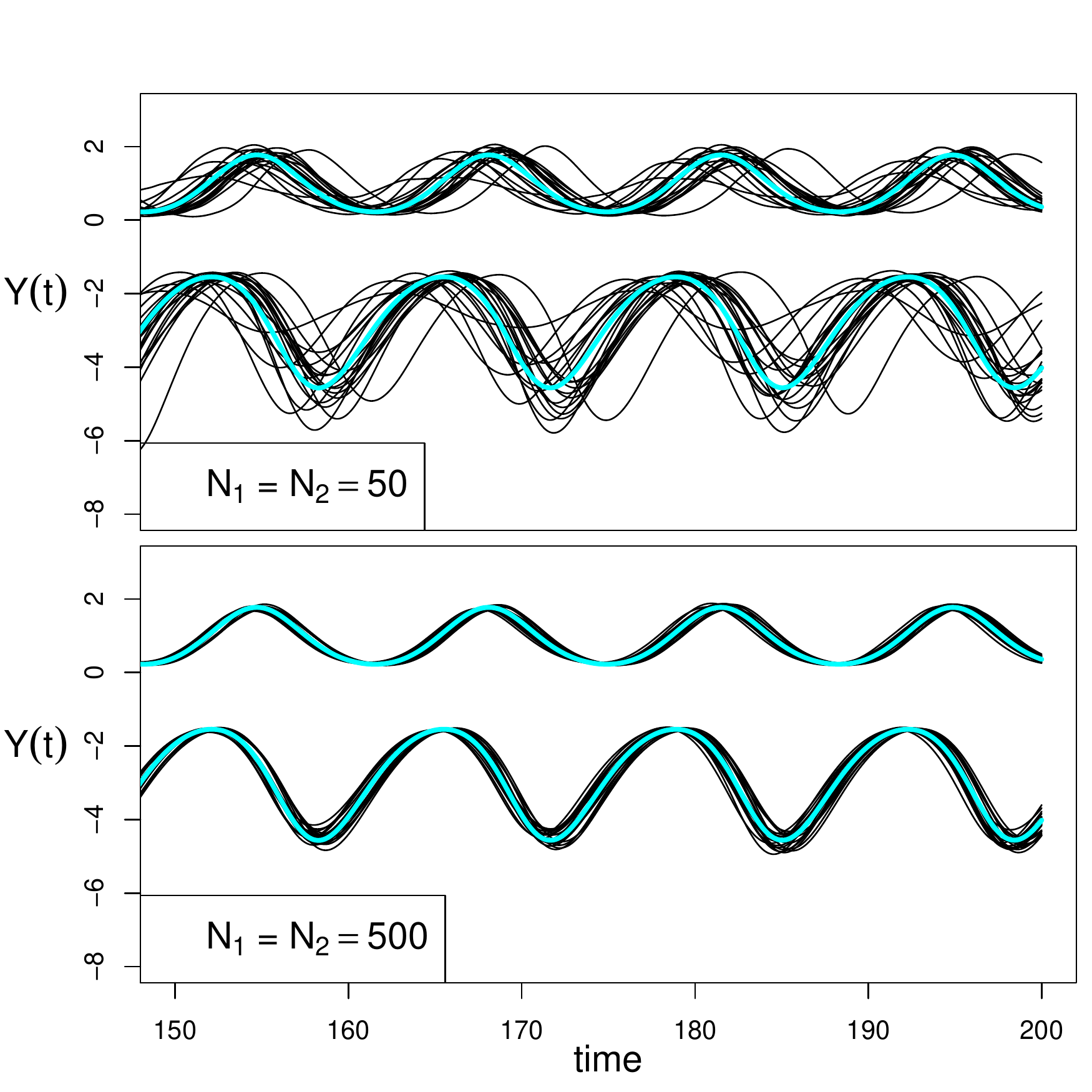}
\caption{ \label{fig:2} Diffusion approximation for increasing $N$, each panel contains 20 realizations. The cyan curve is the limit system. Parameters are as in Fig. \ref{fig:1}. }
\end{figure}

To study phase transitions, we put $\nu=\nu_1=\nu_2=0.8$ and for $\eta_1=\eta_2$ we varied $\kappa=4,8,12,16,20,24$. Condition \eqref{eq:unstable2} is only fulfilled for $6 \leq \kappa \leq 12$. Results are in Fig.\ \ref{fig:3}, left. For $\kappa = 4$ the condition is not fulfilled, and damped oscillations are seen. Then increasing $\kappa$, a phase transition occurs, yielding sustained oscillations. A further phase transition occurs when $\kappa$ becomes larger than 12. For values of $\kappa$ between 12 and 16, damped oscillations happen after the initial large excursion, and when $\kappa$ becomes even larger, the system converges to the steady state in a seemingly monotone manner. 
Note that for $\nu \neq 1$ the steady state depends on the order of the memory.

\begin{figure}[th!]
\centering\includegraphics[width=7cm]{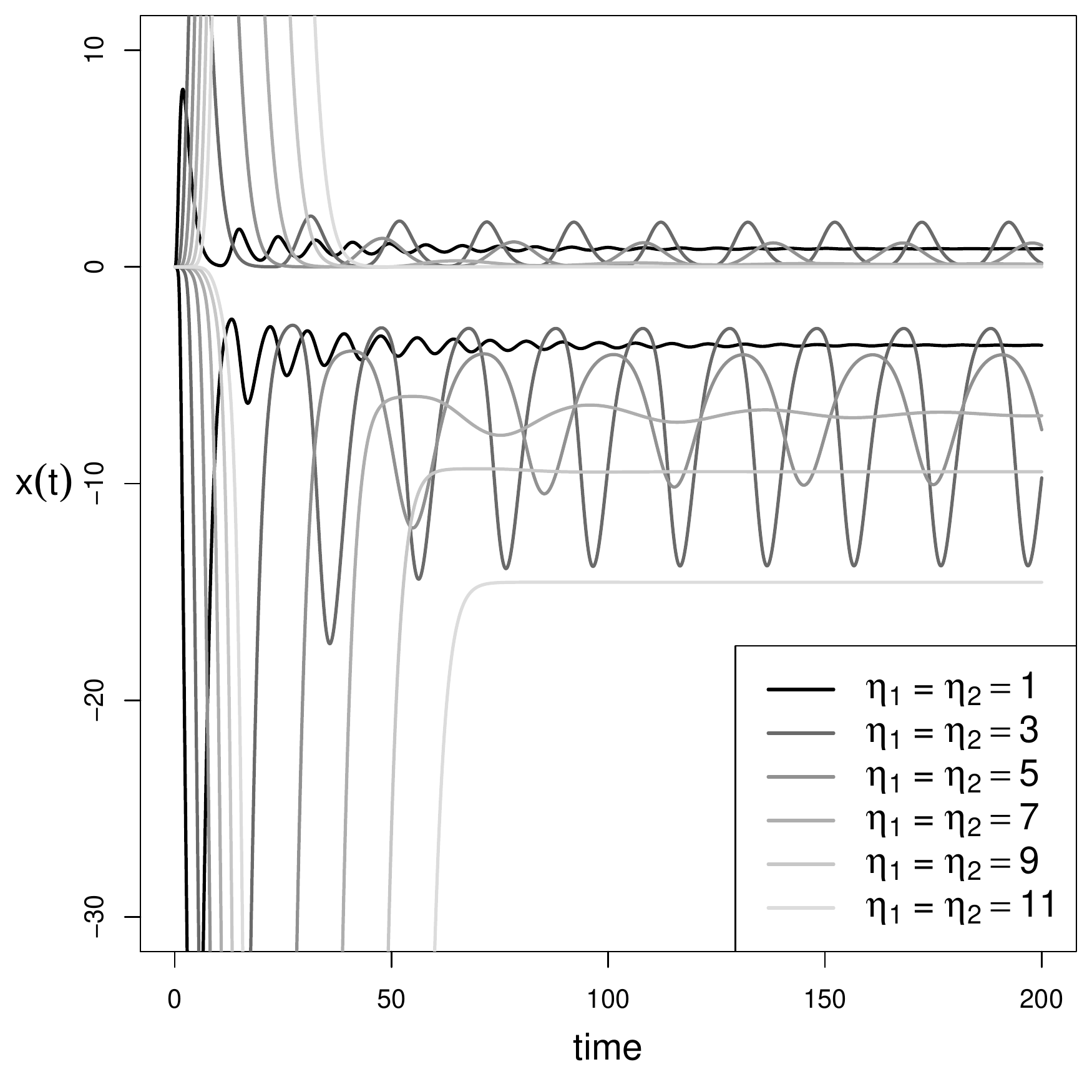} \, \includegraphics[width=7cm]{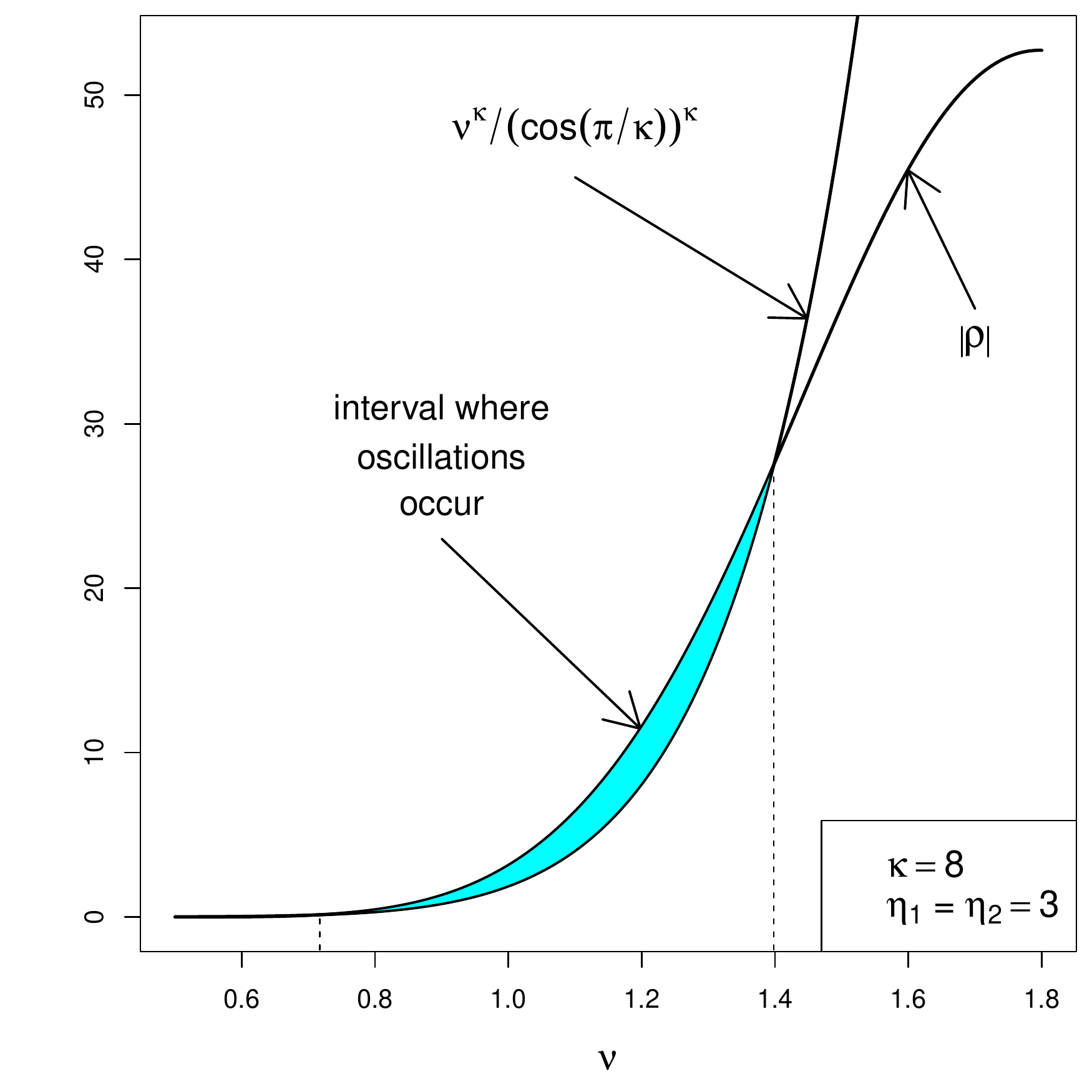}
\caption{ \label{fig:3} Phase transitions and Hopf bifurcations. Left: Phase transitions due to increasing order of the memory. Parameters are $\nu_1=\nu_2=0.8$ and varying $\eta_1=\eta_2$, so that  $\kappa = 2(1+\eta_1)$. Right: Check of condition \eqref{eq:unstable2} as a function of $\nu$. In the cyan area, at least two eigenvalues of $DF(x^*)$ have positive real part, and self-sustained oscillations occur in system \eqref{eq:cascade}. Hopf bifurcations happen at values $\nu_{low}=0.7169$ and at $\nu_{high}=1.3982$. }
\end{figure}

Finally, we fix $\eta_1=\eta_2=3$ and vary $\nu$ in Fig. \ref{fig:3}, right. For low and high values of $\nu$, no oscillations occur, but at $\nu_{low}=0.7169$ a Hopf-bifurcation occurs, and another again at $\nu_{high}=1.3982$. Thus, the system oscillates in an interval $\nu \in (\nu_{low},\nu_{high})$. In general, the interval will depend on the order of the memory.

\section{Appendix}

In the proofs below, $C$ denotes a generic constant, which might change value from equation to equation, and from line to line, even within the same equation.

\subsection{Convolution equations in the matrix case}\label{sec:convolution}

The following is a version of Lemma 26 of \cite{dfh} in the multidimensional case. It is based on old results of \cite{athreya-murthy,crump} on systems of renewal equations.

\begin{lem}[Corollary 3.1 of \cite{crump}, see also Theorem 2.2 of \cite{athreya-murthy}]\label{lem:app}
Let $H(s) $ be given by \eqref{eq:matrixh} such that Assumption \ref{ass:superc} is fulfilled (supercritical case). Put $\Gamma (t) = \sum_{n \geq 1} H^{* n } (t) .$ Then the following assertions hold.\\
1. There is a unique $\alpha_0$ such that $\int_0^\infty e^{- \alpha_0 t } H(t) dt $ has largest eigenvalue $\equiv +1.$ \\
2. $ \Gamma $ is locally bounded. Moreover, for some constant $C,$ 
\begin{equation}\label{eq:growthgamma}
  \Gamma_{i  j}  ( t)  \le  C e^{ \alpha_0 t }. 
\end{equation}
3. For any pair of locally bounded functions $ u , h : \R_+ \to \R^n $ such that $ u = h + H* u, $ it holds $ u = h + \Gamma * h.$ 
\end{lem}

\begin{proof}
1. Due to Assumption \ref{ass:superc}, there exist constants $C$ and $p$ such that $ | h_{kl }| (t) \le C (1+t^p) , $ for all $ 1 \le k, l \le n, $ for all $t \geq 0.$ As a consequence, the matrix-Laplace transform 
$$ {\cal L}_H ( \alpha) := \int_0^\infty e^{- \alpha t } H(t) dt $$
is well-defined for all $ \alpha > 0.$ Being a primitive matrix, by the Perron-Frobenius theorem, it possesses a unique maximal eigenvalue $ \lambda_H ( \alpha ) $ with an associated eigenvector composed of positive coordinates. By Assumption \ref{ass:superc}, $ \lambda_H ( 0) > 1$ and $ \lambda_H ( \infty ) = 0 .$ This implies that there exists a unique $\alpha_0$ such that $ \lambda_H ( \alpha_0) = 1.$ 
(This step demands some extra work, which is done in \cite{athreya-murthy,crump}).

2. Let $ M(t) = e^{- \alpha_0 t } H(t).$ Then $ M* M ( t) = e^{- \alpha_0 t } H* H ( t)$, thus, $ \sum_{n \geq 1 } M^{* n } (t) = e^{- \alpha_0 t } \Gamma (t) .$ Since $H(s) $ is at most of polynomial growth, then $\int_0^\infty t M(t) dt < \infty $ component-wise. This implies that all entries of the matrix ${\cal H} $ of (3.9) of \cite{crump} are finite, so that $ \lim_ {t \to \infty} \sum_{n \geq 1 } (M^{* n } (t))_{i, j }/t $ exists and is finite, for all $ 1 \le i, j \le n$, see page 431 of the proof of Theorem 3.1 of \cite{crump}, see also Theorem 2.2 of \cite{athreya-murthy}. It follows that there exists a constant $C$ such that $ \Gamma_{i, j } (t) \le C e^{\alpha_0 t } , $ for all $ 1 \le i, j \le n.$  

3.\ Follows from Theorem 2.1 of \cite{crump}.
\end{proof}

\subsection{Proof of Theorem \ref{theo:2}}
We first prove Propositions \ref{cor:subc} and \ref{cor:superc}. 

\begin{proof}{\bf of Proposition \ref{cor:subc}}
We have 
$$ \lambda_t^k := \frac{d m_t^k}{ dt}  = f_k \left ( \int_0^t \sum_l h_{kl} (t-s) d m_s^l \right ). $$
Using the Lipschitz property of $f_k, $ we obtain  
\begin{equation}\label{eq:bof0}
 \lambda_t^k \le f_k ( 0 ) + \sum_{l=1}^n L \int_0^t |h_{kl} (t-s)| \lambda_s^l ds .
\end{equation}
Using $H(s)$ defined in \eqref{eq:matrixh}, we can rewrite \eqref{eq:bof0} as
$$ \lambda_t \le f (0) + H * \lambda ( t) , $$
where $ f(0) = (f_1(0) ,\ldots ,  f_n (0) )^T.$ Define $ \Gamma_{l} :=\sum_{1 \le k \le l } H^{* k }.$ Since for any two matrix-valued functions $A,B, $ $ \int_0^\infty A* B (t) dt = (\int_0^\infty A(t) dt ) (\int_0^\infty B(t) dt) ,$ provided the integrals are well-defined, we have 
$$ \int_0^\infty \Gamma_l (t) dt = \sum_{1 \le k \le l} \Lambda^k ,\; \Lambda = \int_0^\infty H(t) dt, $$
having unique maximal eigenvalue $ \sum_{1 \le k \le l } \mu_1^k \le \frac{\mu_1}{1 - \mu_1} .$ As a consequence, the renewal function $\Gamma (t) = \sum_{l \geq 1} H^{* l } (t) $ is well-defined and locally bounded, and the maximal eigenvalue of $\int_0^\infty \Gamma (t) dt $ is given by $\frac{\mu_1}{1 - \mu_1}  .$ 

Any solution $ a(t)$ of $ a = f(0 ) + H * a $ is given by $ a(t) = f(0 ) + \Gamma * f(0) = f(0) +  \int_0^t \Gamma ( s) f(0) ds .$ Therefore, 
$$ \lambda (t) \le  f(0 ) + \left( \int_0^t \Gamma ( s)  ds \right)  f(0)  \le f(0 ) + \left( \int_0^\infty \Gamma ( s) ds \right) f(0) ,$$
where the inequality has to be understood component-wise. Thus, $\lambda ( t) $ is a bounded function of $t.$ This implies the first result.

By \eqref{eq:delta} and \eqref{eq:217}, writing $ \delta^N ( t) = ( \delta^N_1 ( t),\ldots ,  \delta_n^N ( t) )^T, $ 
\begin{eqnarray*}
 \delta^N_k (t) &\le& (H *  \delta^N  )_k (t) +  \frac{C}{\sqrt{N}} \int_0^t \sum_l [\int_0^s  h_{kl}^2 (s-u) \lambda_u^l du ]^{1/2}ds \\
&\le & (H *  \delta^N  )_k (t)  + \frac{Ct }{\sqrt{N}} ,
\end{eqnarray*}
by the first step, since $\lambda_u^l $ are bounded, and since $ h_{kl}$ by assumption are in $L^2( \R_+; \R) .$  As a consequence, 
$$\delta^N_k ( t) \le \frac{Ct }{\sqrt{N}} + \frac{C}{\sqrt{N}} \int_0^t \sum_{l=1 }^n \Gamma_{kl }  ( t- s) s ds  \le \frac{Ct }{\sqrt{N}}.$$

Together with the proof of Theorem \ref{theo:one} this finishes the proof.
\end{proof}

\begin{proof}{\bf of Proposition \ref{cor:superc}}
As in the proof of Proposition \ref{cor:subc}, we obtain \eqref{eq:bof0} for $1 \le k \le n$.  
Hence, using Lemma \ref{lem:app},
$$ \lambda_t^k  \le  f_k  (0) + \sum_l  \left( \int_0^t \Gamma_{kl }  ( s)  ds \right)  f_l (0) \le f_k ( 0) + C e^{\alpha_0 t } \le c e^{\alpha_0 t } , $$
where $c = \max_{1 \le k \le n } f_k ( 0 ) + C.$ 
The second assertion follows then analogously to the proof given for Proposition \ref{cor:subc}.
\end{proof}

The main ingredients for the proof of Theorem \ref{theo:2} are the bounds in Propositions \ref{cor:subc} and \ref{cor:superc}, depending on the criticality. 
We now prove Theorem \ref{theo:2} in the subcritical case, which is an adaptation of the proof of Theorem 10 of \cite{dfh} to the nonlinear case.

\begin{proof}{\bf of Theorem \ref{theo:2}, subcritical case.}
Put $U^N_{k, i } (t) := Z^N_{k, i } (t) - m_t^k, (k,i)\in I^N$, and introduce the martingales
$$
M^N_{k, i } (t)  =
 \int_0^t \int_0^\infty \1_{\left\{ z \le f_k  \left(   \int_0^{s} h_{kl} ( s-u) d m_{l } (u)    \right)\right\} } \tilde N^N_{k,i}  ( ds ,dz) = \bar Z^N_{k,i } ( t) - m_t^k , 
$$
where $\tilde N^N_{k,i}$ denotes the compensated PRM $ N^N_{k,i } (ds, dz ) - ds dz.$ Then 
\begin{equation}\label{eq:u}
U^N_{k, i } (t)  = M^N_{k, i }( t)  +  R_{k, i }^N  (t) ,
\end{equation}
where
$$ R_{k, i}^N ( t)  = Z^N_{k, i } (t)  - \bar Z^N_{k, i } (t)  .$$
We have already shown in \eqref{eq:rem9first} that 
$$  \E ( \sup_{s \le t } | R^N_{k, i } (s) | ) \le C t N^{-1/2} .$$

In \eqref{eq:u}, we have that $[M^N_{k, i }, M^N_{l, j } ]_t = 0 $ for $ (k,i) \neq (l,j) , $ since the martingales almost surely never jump at the same time. 
Moreover, $ [M^N_{k, i }, M^N_{k, i } ]_t = \bar Z^N_{k, i } (t) .$ Finally, $\E [ (M^N_{k, i } (t) )^2 ] = \E ( \bar Z^N_{k, i } (t)  ) = m_t^k   .$ 

We obtain
\begin{eqnarray}\label{eq:tobecited}
  (m_t^{k})^{-1}\E  |U^N_{k, 1 } (t) |  &\le &(m_t^{k})^{-1} \E  |M^N_{k, 1 }| +  (m_t^{k})^{-1} C   t N^{-1 /2}  \nonumber  \\
&\le & \left( m_t^k \right)^{-1/2}  +   \frac{Ct }{m_t^k N^{1/2}}.
\end{eqnarray}

Now we use that $t/N \to 0 $ and $\lim \inf_{t \to \infty } m_t^k /t = \alpha_k > 0 $ 
to deduce that
$$ \lim \! \! \! \! \sup_{ N, t \to \infty } (m_t^{k})^{1/2}\E \left[ |Z^N_{k, 1 } (t)  / m_t^k - 1 |\right] =\lim  \! \! \! \!  \sup_{ N, t \to \infty }  (m_t^{k})^{1/2} \left[ \frac{\E  |U^N_{k, 1 } (t) |}{(m_t^{k})}\right] \le 1 ,$$
implying item 1 of the theorem. 

The proof finishes in the lines of the proof of Theorem 10 of \cite{dfh}. For fixed $k$ and $ i = 1 , \ldots, \ell_k \le N_k ,$ we write 
$$ (m_t^{k})^{1/2} ( Z^N_{k,  i }(t)  / m_t^k - 1 ) = (m_t^{k})^{-1/2} M^N_{k, i } (t)  + (m_t^{k})^{-1/2}  R^{N}_{k, i }  ( t)  .$$ 
But $\E ((m_t^{k})^{-1/2} |\bar R_k^{N} (t) |) \le C t N^{-1/2}(m_t^k)^{-1/2}  \to 0$ as $t , N \to \infty .$ So we only have to prove that $\left( ((m_t^{1})^{-1/2} M^N_{1, i } (t) )_{ 1 \le i \le \ell_1},\ldots , ( (m_t^{n})^{-1/2} M^N_{n, i } (t) )_{1  \le i \le \ell_n} \right)   $ tends in law to ${\cal N} ( 0, I_{\ell_1 +\ldots +  \ell_2}) $ which follows as in \cite{dfh}, proof of Theorem 10. 

\end{proof}

\begin{proof}{\bf of Theorem \ref{theo:2}, supercritical case.}
We use the same notation as in the proof of the subcritical case and obtain in a first step the following control for \eqref{eq:tobecited}.
\begin{equation}\label{eq:1}
 (m_t^{k})^{-1}\E  |U^N_{k, 1 } (t) | \le  \left( m_t^k \right)^{-1/2}  +   \frac{Ce^{\alpha_0 t } }{m_t^k N^{1/2}}.
\end{equation}
Therefore, for $N, t \to \infty $ under the constraint that $ N^{- 1/2}  t^{-1} e^{\alpha_0 t} \to 0, $ 
$$ \lim \! \! \! \! \sup_{ N, t \to \infty } (m_t^{k})^{1/2}\E \left[ |Z^N_{k, 1 } (t)  / m_t^k - 1 |\right] =\lim  \! \! \! \!  \sup_{ N, t \to \infty }  (m_t^{k})^{1/2} \left[ \frac{\E  |U^N_{k, 1 } (t) |}{(m_t^{k})}\right] \le C ,$$
implying item 1 of the theorem.  

The second item of the theorem follows using the same arguments as in the subcritical case. 
\end{proof}

\subsection{Proof of Theorem \ref{theo:approx}}
The proof of Theorem \ref{theo:approx} is based on the following steps. First, a standard calculus shows that we have an approximation result for the generators.

\begin{lem}\label{prop:approx}
Grant the conditions of Theorem \ref{theo:approx}. Then there exists a constant $C$ such that for all $ \varphi \in C^3_b (\R^\kappa , \R ) , $ 
$$ \| A^X \varphi - A^Y \varphi \|_\infty \le C  \frac{\| \varphi \|_{3, \infty }}{N^2 }.$$
\end{lem}

The proof of the above lemma is straightforward and therefore omitted. In a next step, we obtain, applying It\^o's formula with jumps twice, the following estimate.

\begin{lem}\label{lem:2}
Grant the conditions of Theorem \ref{theo:approx}. Then there exists a constant $C$ such that for all $ \varphi \in C^4_b (\R^\kappa , \R ) , $ for any $ \delta > 0, $ 
$$ \| P_\delta^Y \varphi - \varphi - \delta A^Y \varphi \|_\infty \le C \delta^2 \| \varphi \|_{4, \infty }$$
and 
$$ \|P_\delta^X \varphi - \varphi - \delta A^X \varphi \|_\infty \le C \delta^2 \| \varphi \|_{2, \infty }.$$ 
\end{lem}

Again, the proof of the above lemma is straightforward and therefore omitted. Finally, we will use the following fact. 

\begin{lem}\label{lem:3}
Grant the conditions of Theorem \ref{theo:approx}. Then there exists a constant $C$ such that for all $ \varphi \in C^4_b (\R^\kappa  , \R ) , $ for any $t > 0,$
$$ \|P^Y_t \varphi  \|_{4, \infty } \le C \| \varphi \|_{4, \infty }.$$
\end{lem} 

\begin{proof}
By Kunita (1990) \cite{kunita}, see also Ikeda-Watanabe (1989) \cite{ikeda-watanabe}, there exists a version $ \Phi_t ( x) $ of the stochastic flow associated to the SDE \eqref{eq:cascadeapprox} such that $Y^N (t)=  \Phi_t ( x)   ,$ where $Y^N ( t) $ is the solution of 
\eqref{eq:cascadeapprox} starting from $Y^N ( 0) = x.$ Under our assumptions, this flow is a flow of $C^4-$diffeomorphisms. Then we can write 
$ P^Y_t \varphi (x) = \E [\varphi ( \Phi_t (x) ]  ,$ and by dominated convergence, for any $ 1 \le k \le \kappa ,$ $ \partial_{x^{k }} P^Y_t \varphi (x) = \E [ \sum_{l} \partial_{x^{l }} \varphi ( \Phi_t (x) ) \frac{ \partial \Phi^{l} }{ \partial_{x^{k }}} (x) ] .$ The assertion follows then by iterating this argument, using classical estimates on the derivatives $ \partial^\alpha \Phi_t (x) $ obtained e.g.\ in \cite{ikeda-watanabe}.
\end{proof}

We are now able to finish the proof of Theorem \ref{theo:approx}. 

\begin{proof}{\bf of Theorem \ref{theo:approx}}
Fix $\delta > 0 $ and write $t_k = k \delta \wedge t , $ for $ k \geq 0.$ A standard trick yields 
\begin{equation}\label{eq:good}
 \| P_t^X \varphi - P_t^Y \varphi \|_\infty \le \sum_{k : k \delta \le t } \| P_{t - t_{k+1}}^X \Delta_\delta P_{t_k}^Y \varphi \|_\infty \le \sum_{k : k \delta \le t } \| \Delta_\delta P_{t_k}^Y \varphi \|_\infty ,
\end{equation}
where we define $ \Delta_\delta \varphi (x) = P_\delta^X \varphi ( x) - P_\delta^Y \varphi ( x) .$ Since $ \varphi \in C^4_b, $ also $ P_{t_k}^Y  \varphi \in C^4_b .$ 

By Lemma \ref{lem:2}, 
$$  \| P_\delta^X \varphi ( x) - P_\delta^Y \varphi ( x)\|_\infty \le \delta \| A^X \varphi - A^Y \varphi \|_\infty + \delta^2 \| \varphi\|_{4, \infty }.$$
Using Lemma \ref{prop:approx}, we deduce that 
$$  \| P_\delta^X \varphi ( x) - P_\delta^Y \varphi ( x)\|_\infty \le [\delta C \frac{1}{N^2 } + \delta^2 ]   \| \varphi\|_{4, \infty }.$$ 
Together with Lemma \ref{lem:3} and \eqref{eq:good}, this yields
$$ \| P_t^X \varphi - P_t^Y \varphi \|_\infty \le C (\frac{1}{N^2 } + \delta ) \; \| \varphi \|_{4, \infty } \left( \sum_{k : k \delta \le t } \delta  \right)  \le C t \frac{1}{N^2 } \| \varphi\|_{4, \infty }, $$
where the last inequality follows by choosing $ \delta = \frac{1}{N^2 } $ and using that $ card \{ k : k \delta \le t \} \le \frac{t}{\delta} .$  

\end{proof}

\subsection{Control theorem and proof of Proposition \ref{prop:control}}\label{sec:controlproof}
We will use the control theorem which goes back to Strook and Varadhan (1972) \cite{StrVar-72}, see also Millet and Sanz-Sole (1994) \cite{MilSan-94}, theorem 3.5, in order to prove Proposition \ref{prop:control}.   

For some time horizon $T_1<\infty$ which is arbitrary but fixed, write $\,\tt H\,$ for the Cameron-Martin space of measurable functions ${\tt h}:[0,T_1]\to \R^2 $ having absolutely continuous components ${\tt h}^\ell(t) = \int_0^t \dot{\tt h}^\ell(s) ds$ with $\int_0^{T_1}[\dot{\tt h}^\ell]^2(s) ds < \infty$, $1\le \ell\le 2$. For $x\in \R^\kappa $ and ${\tt h}\in{\tt H}$, consider the deterministic system 
\begin{equation}\label{generalcontrolsystem}
\varphi = \varphi^{(N, {\tt h}, x)} \; \mbox{solution to}\; d \varphi (t) = b (  \varphi (t) ) dt + \frac{1}{\sqrt{N}} \sigma( \varphi (t) ) \dot{\tt h}(t) dt, \; \mbox{with $\varphi (0)=x,$}  
\end{equation}
on $[0, T_1 ].$ Thus $\varphi $ is a function $[0,T_1 ]\to \R^\kappa $.

%
Using localization techniques as in H\"opfner, L\"ocherbach and Thieullen (2015) \cite{hh3}, Theorem 5, we obtain the following result. 

\begin{prop}\label{theo:4bis}
Grant the assumptions of Theorem \ref{theo:orbit}. Denote by $ Q_x^{t_1} $ the law of the solution $ (Y^N(t))_{0 \le t \le t_1} $ of  \eqref{eq:cascadeapprox}, starting from $Y^N (0) = x .$ Let $\,\varphi = \varphi^{(N, {\tt h},x)}\,$ denote a solution to
$$d \varphi (t) \;=\; b ( \varphi (t) )\, dt \;+\; \frac{1}{\sqrt{N}} \sigma ( \varphi  (t) )\, \dot{\tt h}(t)\, dt   \quad,\quad \varphi (0)=x.
$$

Fix $ x \in K  $ and  $  {\tt h } \in \tt H $ such that $\,\varphi = \varphi^{(N, {\tt h},x)}\,$ exists on some time interval $ [ 0, T_1 ] $ for $T_1 > t_1.$ 
Then
$$ \left(\varphi^{(N, {\tt h} , x ) }\right)_{| [0,  t_1 ] } \in \overline{ {\rm supp} \left( Q_x^{t_1 } \right)}.$$
\end{prop}
 
We now show how to use Proposition \ref{theo:4bis} in order to prove Proposition \ref{prop:control}. 

\begin{proof}{\bf of Proposition \ref{prop:control}}
Fix $ x \in \R^\kappa  $ and $t_1 > 1.$ Recall that, to simplify notation, we write $ x = ( x^1 , \ldots , x^\kappa ), $ where $ \kappa = 2 + \eta_1 + \eta_2 $  instead of $x=  (x^{1, 0 }, \ldots , x^{1, \eta_1}, x^{2, 0}, \ldots , x^{2, \eta_2} ) .$ In particular, the coordinates $ x^{\eta_1 +1 }$ and $ x^\kappa $ correspond to the two coordinates which are driven by Brownian noise. Let $\Gamma ( t) , t \in [0, T ] , $ be a parametrization of the periodic orbit, with $T$  the periodicity of $\Gamma .$ 

Now, we choose a $\cal{C}^\infty$-function $\,\gamma   =  (\gamma^1 , \gamma^2 ) : \R_+ \to \R^2 $ satisfying 
\begin{equation}\label{prescription-gamma1}
\left\{\begin{array}{l}
\gamma^1 (0)= x^{ 1 }\\
\gamma^2 (0) =x^{\eta_1 + 2} \\ 
\gamma (s) \equiv  (\Gamma^{ 1 } (s) , \Gamma^{\eta_1 +2  } (s) )  \;\;\mbox{on $[1 ,\infty).$}
\end{array}\right. 
\end{equation}
We want to use $\gamma $ as a smooth trajectory driving the two components $\varphi^1 $ corresponding to $Y^N_{1, 0} $ and $\varphi^{\eta_1 +2}$ corresponding to $ Y^N_{2, 0} $ from their initial position to a position on the periodic orbit, during a time period of length one. 

We now show that it is indeed possible to choose a control $ {\tt h } $ such that $\varphi^1 = \gamma^1$ and $ \varphi^{\eta_1 +2} = \gamma^2 .$ Recall that the diffusion coefficient $ \sigma $ is null on every coordinate except the coordinates $ \eta_1 +1 $ and $\kappa .  $ As a consequence, any choice of $ {\tt h } $ does only allow to influence directly these two coordinates $ \eta_1 +1 $ and $\kappa .  $ However, above we have prescribed a trajectory $\gamma $ to the two coordinates $ 1$ and $ \eta_1 +2.$ So we have to prove that such a choice of $ {\tt h } $ is possible. 

Suppose for a moment that we have already found this control ${\tt h }.$ 
Then, by the structure of $b,$ once $\varphi^1 $ and $\varphi^{\eta_1 + 2 } $ are fixed, we necessarily have
\begin{equation}\label{eq:necessaryorbit1}
 \varphi^2 ( t) = \frac{ d \varphi^1 (t)}{dt} + \nu_1 \varphi^1 (t) , \; \ldots \; , \varphi^{\eta_1 + 1} (t) = \frac{ d \varphi^{\eta_1}  (t)}{dt} + \nu_1 \varphi^{\eta_1}  (t)
\end{equation} 
for the first population, and 
\begin{equation}\label{eq:necessaryorbit2}
 \varphi^{\eta_1+3} ( t) = \frac{ d \varphi^{\eta_1+2} (t)}{dt} + \nu_2 \varphi^{\eta_2 +2} (t) ,\;  \ldots \; , \varphi^{\kappa } (t) = \frac{ d \varphi^{\kappa - 1}  (t)}{dt} + \nu_2 \varphi^{\kappa - 1 }  (t) .
\end{equation} 
In other words, once $ \varphi^1 $ and $\varphi^{\eta_1 +2}$ are fixed, all other coordinates are entirely determined as measurable functions of $\gamma.$ Moreover, by the structure of the equations given in \eqref{eq:cascade}, it is clear that for all $t \geq 1, \varphi (t) = \Gamma ( t),$ i.e. the trajectory evolves on the orbit after time $1.$ 

We have to show that we can indeed find a function $ {\tt h } $ which allows for the above choice of $\varphi.$ The control ${\tt h } $ is related to the coordinates $ \varphi^{\eta_1 +1} $ and $ \varphi^\kappa $ through 
$$ \frac{ d \varphi^{\eta_1+1}  (t)}{dt}  = - \nu_1 \varphi^{\eta_1+1}  (t) + c_1 f_2 ( \varphi^{\eta_1+2}  (t)) + \frac{ c_1 }{ \sqrt{p_2}} \sqrt{ f_2 ( \varphi^{\eta_1 + 2} (t) ) } \dot{\tt h}^1 (t)  $$
and 
$$ \frac{ d \varphi^{\kappa}  (t)}{dt}  = - \nu_2 \varphi^{\kappa}  (t) + c_2 f_1 ( \varphi^{1}  (t)) + \frac{ c_2 }{ \sqrt{p_1}} \sqrt{ f_1 ( \varphi^{1} (t) ) } \dot{\tt h}^2 (t)  .$$
In the above two formulas, all functions $\varphi$ are known as measurable functions of the prescribed trajectory $ \gamma ,$ and $\dot{\tt h}^1 (t) $ and $\dot{\tt h}^2 (t) $ have to be chosen. But since $f_1 $ and $f_2$ are strictly positive, there exists indeed ${\tt h}:[0, \infty [ \to \R^2 $ achieving this choice of $\gamma .$ It suffices to choose
\begin{equation}\label{eq:h1}
\dot{\tt h}^1 (t) = \frac{ \frac{ d \varphi^{\eta_1+1}  (t)}{dt}  + \nu_1 \varphi^{\eta_1+1}  (t) -c_1 f_2 (\varphi^{\eta_1+2}  (t)) }{ c_1 / \sqrt{p_2} \sqrt{ f_2 ( \varphi^{\eta_1 + 2} (t) ) }} 
\end{equation}
and 
\begin{equation}\label{eq:h2}
 \dot{\tt h}^2 (t) = \frac{ \frac{ d \varphi^{\kappa }  (t)}{dt}  + \nu_1 \varphi^{\kappa }  (t) - c_2 f_1 ( \varphi^{1} (t))  }{ c_2 / \sqrt{p_1} \sqrt{ f_1 ( \varphi^{1} (t) ) }} .
\end{equation} 
Since $f_1 $ and $f_2$ are strictly positive (and even lower bounded on $K$), $\dot{\tt h } (t) $ is well-defined. 

Now, notice that for all $t \geq 1, $ $\varphi^1 (t) = \Gamma^1 (t) $ and $\varphi^{\eta_1+2 } ( t) = \Gamma^{\eta_1 + 2 } ( t )$ are evolving on the periodic orbit. Hence by \eqref{eq:necessaryorbit1} and \eqref{eq:necessaryorbit2}, necessarily $\varphi (t) = \Gamma (t) $ for all $t \geq 1. $ In particular, $\dot{\tt h}^1 (t) = \dot{\tt h}^2 (t) = 0 $ for all $t > 1.$ 

Hence, we have constructed a control forcing the trajectory to be on the periodic orbit after a fixed time. By construction, $ S ( \varepsilon, \varphi ) := \{ \psi \in C( \R_+ ; \R^\kappa ) : | \varphi (t) - \psi ( t) | < \varepsilon \; \forall 1 \le t \le t_1 \}  \subset O .$ Moreover, by Proposition \ref{theo:4bis}, $ \varphi_{ | [0, t_1]}  \in \overline{ {\rm supp} \left( Q_x^{t_1 } \right)} ,$ whence $ Q_x^{t_1} ( S ( \varepsilon, \varphi ) ) > 0 $ and thus a fortiori $Q_x^{t_1} (O) > 0, $ implying the assertion for all  $t_1 > 1.$  
 
\end{proof} 

%

\section*{Acknowledgments}
This research has been conducted as part of the project Labex MME-DII (ANR11-LBX-0023-01).  Author EL thanks Michel Bena\"{\i}m for many stimulating discussions concerning monotone cyclic feedback systems. We thank Mads Bonde Raad for careful reading and critical remarks. We also thank two anonymous referees for helpful comments and suggestions.  The work is part of the Dynamical Systems Interdisciplinary Network, University of Copenhagen. Moreover, it is part of the activities of FAPESP Research, Dissemination and Innovation Center for Neuromathematics (grant
2013/07699-0, S.\ Paulo Research Foundation).

\end{document}